\documentclass{amsart}
\usepackage{verbatim,amsmath,amssymb,amscd,color,graphics}
\usepackage[enableskew]{youngtab}
\usepackage[dvips]{graphicx}

\usepackage[colorlinks=true, pdfstartview=FitV, linkcolor=blue, citecolor=blue, urlcolor=blue]{hyperref}

\numberwithin{equation}{section}
 
\newtheorem{theorem}{Theorem}
\newtheorem{proposition}[theorem]{Proposition}

\newtheorem{conjecture}[theorem]{Conjecture}

\theoremstyle{definition}

\newtheorem{definition}[theorem]{Definition}
\newtheorem{remark}[theorem]{Remark}
\newtheorem{example}[theorem]{Example}
 
\numberwithin{theorem}{section}

\newcommand{\cc}{\operatorname{cc}}
\newcommand{\Conf}{C}
\newcommand{\Confb}{\overline{C}}
\newcommand{\Db}{\overline{D}}
\newcommand{\fillmap}{\operatorname{fill}}
\newcommand{\gammarc}{\gamma_{\operatorname{rc}}}
\newcommand{\geh}{\mathfrak{g}}
\newcommand{\HH}{\mathcal{H}}
\newcommand{\J}{I_0}
\newcommand{\ku}{{}}
\newcommand{\La}{\Lambda}
\newcommand{\mone}{\bar{1}}
\newcommand{\mtwo}{\bar{2}}
\newcommand{\mthree}{\bar{3}}
\newcommand{\mfour}{\bar{4}}

\newcommand{\RC}{\operatorname{RC}}
\newcommand{\RCb}{\overline{\RC}}
\newcommand{\sigmarc}{\sigma_{\operatorname{rc}}}
\newcommand{\Z}{\mathbb{Z}}

\begin{document}
 
\title{Affine crystal structure on rigged configurations of type $D_n^{(1)}$}

\author[M.~Okado]{Masato Okado}
\address{Department of Mathematical Science,
Graduate School of Engineering Science, Osaka University,
Toyonaka, Osaka 560-8531, Japan}
\email{okado@sigmath.es.osaka-u.ac.jp}

\author[R.~Sakamoto]{Reiho Sakamoto}
\address{Department of Physics, Tokyo University of Science, Kagurazaka, Shinjukuku, 
Tokyo 162-8601, Japan}
\email{reiho@rs.tus.ac.jp}
 
\author[A.~Schilling]{Anne Schilling}
\address{Department of Mathematics, University of California, One Shields
Avenue, Davis, CA 95616-8633, U.S.A.}
\email{anne@math.ucdavis.edu}
\urladdr{http://www.math.ucdavis.edu/\~{}anne}
 

\begin{abstract}
Extending the work in \cite{S:2006}, we introduce the affine crystal action
on rigged configurations which is isomorphic to the Kirillov--Reshetikhin
crystal $B^{r,s}$ of type $D_n^{(1)}$ for any $r,s$. We also introduce a
representation of $B^{r,s}$ ($r\ne n-1,n$) in terms of tableaux of rectangular shape $r\times s$,
which we coin Kirillov--Reshetikhin tableaux
(using a non-trivial analogue of the type $A$ column splitting procedure)
to construct a bijection between elements of a tensor product
of Kirillov--Reshetikhin crystals and rigged configurations.
\end{abstract}
 
\maketitle
 
\tableofcontents
 
\section{Introduction}

Motivated by studies using the Bethe Ansatz, Kerov, Kirillov and Reshetikhin~\cite{KKR:1986} introduced interesting 
new combinatorial objects coined rigged configurations (RCs), and found a bijection to semistandard tableaux. It was later 
realized~\cite{KSS:2002} that this type of bijection can be extended to a bijection between RCs and elements in a multiple
tensor product of Kirillov--Reshetikhin (KR) crystals of type $A$ satisfying highest weight conditions.
Since there is an action of Kashiwara operators on crystals, it is natural to try to find the action
on RCs through the bijection. This was achieved by the third author in~\cite{S:2006} for $f_i$ and $e_i$ with
$i\ne0$, and subsequently in~\cite{SW:2010} for $i=0$ by considering the action of the promotion 
operator on RCs.

We wish to consider a similar problem for type $D_n^{(1)}$. Since the multiplicities of irreducible components
in a multiple tensor product turn out large, we start by considering the single KR crystal $B^{r,s}$
in this paper. For classical Kashiwara operators $f_i$ and $e_i$ ($i\ne 0$) the action has already been provided 
in~\cite{S:2006}. The action of $f_0$ and $e_0$ is defined by 
\[
	f_0 = \sigma \circ f_1 \circ \sigma \quad \text{and} \quad e_0 = \sigma \circ e_1 \circ \sigma
\]
using the involution $\sigma$ corresponding to exchanging the Dynkin nodes $0$ and $1$. Since $\sigma$
commutes with $f_i$ and $e_i$ for $i=2,\ldots,n$, we only need to define the action of $\sigma$ on RCs which
are $\{2,\ldots,n\}$-highest weight. This is performed by defining a map $\gammarc$ from $\pm$-diagrams,
another combinatorial object parameterizing $\{2,\ldots,n\}$-highest weight elements, to 
$\{2,\ldots,n\}$-highest weight RCs. This map has an interesting feature. We cut a given $\pm$-diagram
into columns, associate to each column an atomic RC, and add up these atomic RCs in a certain way. In other
words, $\gammarc$ can be interpreted as a linear map.
Such a nice relationship between rigged configurations and
$\pm$-diagrams was originally suggested through the analysis of the expression of
the combinatorial $R$-matrix~\cite{OSaka:2010} written in terms of
$\pm$-diagrams via the expected property that the combinatorial $R$-matrix
acts as the identity on rigged configurations (Conjecture~\ref{conjecture.R}).

Combining these results we obtain our first main theorem (Theorem~\ref{theorem.main}) describing the affine 
crystal structure on rigged configurations corresponding to a single Kirillov--Reshetikhin crystal of type $D_n^{(1)}$.
In addition, we show that the associated crystal isomorphism preserves the grading by energy and cocharge
(Theorem~\ref{theorem:cc=D_for_single_rectangle}).
This sheds new light on the crystal structure of KR crystals.
As discussed above, the core of the construction of $e_0$ and $f_0$ is
the bijection $\gamma$~\cite{S:2008} between $\pm$-diagrams and $\{2,3,\ldots,n\}$-highest weight
Kashiwara--Nakashima tableaux~\cite{KN:1994}.
However, the bijection $\gamma$ requires a non-trivial algorithm as described in Proposition~\ref{proposition.gamma}.
According to our results, $\pm$-diagrams are related to rigged configurations by a
linear operation $\gammarc$.
Thus it is tempting to regard $\pm$-diagrams and rigged configurations as having a common mathematical origin.

We remark that these results can be viewed as another important example of significant properties
of rigged configurations with respect to deep structures of the underlying algebra.
For example, in~\cite{OSaka:2011} an interesting new bijection related to rigged configurations
and Littlewood--Richardson tableaux is introduced which is expected to be an analogue of the involution corresponding
to exchanging Dynkin nodes 0 and $n-1$ constructed in~\cite{LOS:2011}.
Another such phenomenon is that generalizations of Sch\"{u}tzenberger's involution become
simple operations on rigged configurations (taking complements of the riggings, see~\cite{SS:X=M}).
Also the complicated action of the combinatorial $R$-matrix becomes trivial on
rigged configurations (see~\cite{SS:X=M}).
This property plays a key role in the recently discovered connection with
a discrete integrable system called the box-ball system (see \cite{KOSTY:2006}).
Hence it is desirable to find a description of the bijection $\Phi$
between rigged configurations and elements of tensor product of Kirillov--Reshetikhin crystals
in an explicit way.

This brings us to the next purpose of the present paper. We provide an explicit combinatorial algorithm
for $\Phi$ (Section~\ref{section.cb}), which leads to a definition of a completely new set of tableaux
which we coin Kirillov--Reshetikhin tableaux (see Section~\ref{section.filling map} for the definition).
For type $A$ the bijection from rigged configurations to Kirillov--Reshetikhin crystals is given by successively 
applying a fundamental algorithm $\delta$.
Each application of $\delta$ produces a letter, which can be placed in the $r\times s$ rectangle corresponding to the
Kirillov--Reshetikhin crystal $B^{r,s}$ resulting into a semistandard tableau.
The algorithm for $\delta$ also exists for type $D$~\cite{OSS:2003a} for tensor products of the Kirillov--Reshetikhin
crystal associated to the vector representation. 
We extend this to arbitrary Kirillov--Reshetikhin crystals $B^{r,s}$, which produces tableaux whose shape is 
an $r\times s$ rectangle. We note that these tableaux are completely different from the usual
Kashiwara--Nakashima tableaux~\cite{KN:1994} since in their representation tableaux do not necessarily have 
rectangular shape.
The correspondence between Kashiwara--Nakashima and Kirillov--Reshetikhin tableaux
for the highest weight elements is given by a map called the filling map
and is extended to arbitrary elements via an isomorphism of crystals (see Definition~\ref{def:KRtableaux}).
In Theorem~\ref{theorem:single_highest_case} we show that for the classically highest weight
elements in a single Kirillov--Reshetikhin crystal the combinatorial definition of the bijection
agrees with the correspondence under the affine crystal isomorphism.
At the end of the paper we state several conjectures (Conjectures~\ref{conjecture.broccoli}, \ref{conjecture.R},
and~\ref{conjecture.energy charge}) that provide evidence that our new tableaux representation is natural and useful.

We remark that our strategy for defining a new tableau representation for Kirillov--Reshetikhin crystals
is, in principle, not limited to type $D^{(1)}_n$.
Indeed, there are several extensions of the combinatorial algorithm $\delta$ for arbitrary non-exceptional affine algebras
\cite{OSS:2003a} as well as type $E^{(1)}_6$ \cite{OSano}.
This forms another motivation for the study of the combinatorial bijection between
rigged configurations and Kirillov--Reshetikhin crystals.

The paper is organized as follows. In Section~\ref{section.crystal background} we review facts about crystal
bases that are needed for this paper. Rigged configurations and background material are presented in
Section~\ref{section.rc}. Section~\ref{section.affine structure} contains the main results, namely the
affine crystal structure on rigged configurations for a single tensor factor for type $D_n^{(1)}$.
The combinatorial bijection and associated conjectures are the subject of Section~\ref{section.cb}.

\subsection*{Acknowledgments}
MO is partially supported by the Grants-in-Aid for Scientific Research No.
23340007 and No. 23654007 from JSPS.
RS is partially supported by Grants-in-Aid for Scientific Research No. 21740114 from JSPS.
AS is in part supported by NSF grants DMS--0652641 and DMS--1001256.

\section{Background on crystals}
\label{section.crystal background}

In this section we review some facts needed about crystal bases.

\subsection{Review of crystals and notation}

Crystal theory was introduced by Kashiwara~\cite{Kashiwara:1994} and provides a combinatorial approach
in terms of tableaux to the representation theory of quantum groups and Lie algebras.
A crystal is a nonempty set $B$ together with Kashiwara lowering and raising operators
$f_i$ and $e_i$ for $i\in I$, where $I$ is the index set of the Dynkin diagram of the associated
Lie algebra $\geh$. The Kashiwara operators are the $q\to 0$ limits of the Chevalley operators of the
corresponding quantum algebra $U_q(\geh)$. One of the amazing properties of crystals is the
fact that they are well-behaved with respect to tensor products. Given two $\geh$-crystals $B$ and $B'$,
the Kashiwara operators on the tensor product $B\otimes B'$ can be described by a completely
combinatorial rule called the signature rule.
For an introduction to crystal theory see for example the book by Hong and Kang~\cite{HK:2002}.

For an affine Kac--Moody algebra $\geh$, we denote by $\geh_0$ the finite-dimensional simple Lie algebra
obtained by removing the $0$ node from the Dynkin diagram of $\geh$ and by $\alpha_i$ ($i\in I$) the 
simple roots. We also denote by $\varpi_i$ ($i\in I_0:=I\setminus\{0\}$) the fundamental weights of $\geh_0$. Let $\Lambda$ be a dominant weight of $\geh_0$.
For crystals $B(\Lambda)$ associated to highest weight representations of highest weight
$\Lambda$ of $U_q(\geh_0)$, there exist generalizations of the usual semistandard Young tableaux
(which we can think of as type $A$ objects) known as Kashiwara--Nakashima (KN) tableaux~\cite{KN:1994}.
For type $D_n$ these are tableaux of shape $\Lambda$ over some ordered alphabet $\{1< 2 < \ldots < n,\overline{n} <
\ldots < \overline{2} < \overline{1}\}$. Here the letters $n$ and $\overline{n}$ are incomparable.
For the precise definition of the semistandard condition for type $D_n$ see~\cite{HK:2002}.
Let $B$ be a crystal of type $\geh$ and $b\in B$. For a subset $J\subset I$ we say that $b$ is $J$-\textbf{highest weight}
if $e_i(b) =0$ for $i\in J$. We say $b$ is \textbf{highest weight} if it is $\J$-highest weight.

In this paper we freely identify dominant weights (without spin nodes) and partitions. More precisely, given a partition 
$\lambda = (\lambda_1\ge \lambda_2 \ge \cdots \ge \lambda_\ell)$ with at most $n-2$ parts
(that is $\ell\le n-2$), we can associate the dominant weight $\Lambda = \varpi_{i_1} + \cdots + \varpi_{i_k}$
where the $i_j$ for $1\le j\le k=\lambda_1$ are the heights of the columns in $\lambda$. When we draw
the Ferrers diagram for the partition $\lambda$ with $\lambda_i$ boxes in row $i$, we use English notation 
adjusting the rows on the left and placing the largest part on the top. 
The height of a cell in a partition is equal to its row index (that is the distance from the top of the Ferrers diagram).
We also use English convention for tableaux.

\subsection{$\pm$-diagrams and definition of $\sigma$} \label{section.sigma}

In order to define Kirillov--Reshetikhin crystals for type $D_n^{(1)}$ following~\cite{FOS:2009,S:2008}, we need to define
an involution $\sigma$ which corresponds to the type $D_n^{(1)}$ Dynkin diagram automorphism of interchanging
nodes 0 and 1. This is achieved by noting that $\sigma$ commutes with the Kashiwara crystal operators $f_i$ and $e_i$
for $i\in \{2,3,\ldots,n\}$. Then $\sigma$ is defined explicitly on $\{2,3,\ldots,n\}$-highest weight vectors.

It turns out that $\{2,3,\ldots,n\}$-highest weight vectors are in bijection with so-called $\pm$-diagrams
(see Proposition~\ref{proposition.gamma} below). A \textbf{$\pm$-diagram} $P$ is a sequence of shapes
$\lambda \subset \mu \subset \Lambda$ such that $\Lambda/\mu$ and $\mu/\lambda$ are horizontal strips
(i.e. every column contains at most one box). We depict this $\pm$-diagram by the skew tableau of shape
$\Lambda/\lambda$ in which the cells of $\mu/\lambda$ are filled with the symbol $+$ and those of
$\Lambda/\mu$ are filled with the symbol $-$. The partition $\Lambda$ is called the outer shape of $P$ and
$\lambda$ is called the inner shape of $P$. In this paper we only require $\pm$-diagrams for the nonspin case,
that is, when the height of $\Lambda$ is at most $n-2$.

For our purposes it will be convenient to state the bijection between $\pm$-diagrams and $\{2,3,\ldots,n\}$-highest
weight elements in an inductive fashion.

\begin{proposition}
\label{proposition.gamma}
\rm{(\cite{S:2008},\cite[Section 3.2]{FOS:2009})}
There is a bijection $\gamma$ from $\pm$-diagrams of outer shape $\Lambda$ to $\{2,3,\ldots,n\}$-highest weight elements
in the highest weight crystal $B(\Lambda)$. 
The $\pm$-diagram $P$ which has $+$ in every column and no $-$ corresponds to the highest weight vector $u\in B(\Lambda)$ 
of weight equal to the outer shape of $P$. Given a $\pm$-diagram $P$ we can obtain the corresponding 
$\{2,3,\ldots,n\}$-highest element $\gamma(P)=b$ inductively as follows:
\begin{enumerate}
\item[\textbf{Case 1:}] $P$ has a column where a $+$ can be added.\newline
Let $P'$ be the $\pm$-diagram obtained from $P$ by adding a $+$ in the rightmost possible column at height $h$.
Then $b = f_1 f_2 \cdots f_h \gamma(P')$.
\item[\textbf{Case 2:}] $P$ has no column where a $+$ can be added and at least one $-$.\newline
Let $P'$ be the $\pm$-diagram obtained from $P$ by removing the leftmost $-$ at height $h$ and either moving the $+$
in the same column down if $h>1$ or adding a $+$ if $h=1$.
Then $b = f_1 f_2 \cdots f_n f_{n-2} f_{n-3} \cdots f_h \gamma(P')$.
\end{enumerate}
\end{proposition}

\begin{example} Let $n=5$ and 
$$P=\,\Yvcentermath1\young(\ku\ku\ku,\ku+,\ku-)\, .$$
Then according to the inductive procedure of Proposition~\ref{proposition.gamma} we have
\begin{equation*}
\begin{split}
	\gamma \left( \,\Yvcentermath1\young(\ku\ku\ku,\ku+,\ku-)\, \right) 
	&= f_1 \, \gamma \left( \,\Yvcentermath1\young(\ku\ku +,\ku+,\ku-)\, \right) 
	= (f_1) (f_1 f_2 f_3) \, \gamma \left( \,\Yvcentermath1\young(\ku\ku +,\ku+,+-)\, \right) \\[1mm]
	&= (f_1) (f_1 f_2 f_3) (f_1f_2 f_3 f_4 f_5 f_3)\, \gamma \left( \,\Yvcentermath1\young(\ku\ku +,\ku\ku,++)\, \right)\\[1mm]
	&= (f_1) (f_1 f_2 f_3) (f_1f_2 f_3 f_4 f_5 f_3) \, \Yvcentermath1\young(111,22,33)
	= \Yvcentermath1\young(122,23,4\mone) \; .
\end{split}
\end{equation*}
\end{example}

We now define the following map $\mathfrak{S}$ on $\pm$-diagrams.
Let $c_.(h)$, $c_+(h)$, $c_-(h)$, $c_\pm(h)$ be the number of columns in $P$ of outer height $h$ 
with no sign, $+$, $-$, $\pm$, respectively. As we will see in~\eqref{equation.KR decomposition} below, in our setting
the values of $h$ are either all even or all odd. Note that $P$ is specified if all values $c_.(h)$, $c_+(h)$, $c_-(h)$, $c_\pm(h)$
are given.

\begin{definition} \label{definition.sigma pm}
Let $P$ be a $\pm$-diagram of outer shape $\Lambda$, where the columns in $\Lambda$ are either all of even or
all of odd height. Then $\mathfrak{S}(P)$ is the $\pm$ diagram, where compared to $P$ the values $c_+(h)$ and $c_-(h)$
are interchanged for $h\ge 1$, and the values of $c_.(h-2)$ and $c_\pm(h)$ are interchanged for $h\ge 2$.
\end{definition}

\begin{example}
Let $n\ge 5$ and 
$$P=\,\Yvcentermath1\young(\ku\ku++-,\ku+,\ku-)\, .$$
In this case $c_.(3)=c_\pm(3)=c_-(1)=1$, $c_.(1)=0$, and $c_+(1)=2$. Then
$$\mathfrak{S}(P)=\,\Yvcentermath1\young(\ku\ku+--,\ku,\ku,+,-)\, .$$
We note here a subtle point which will become relevant in Definition~\ref{definition.sigma}. Namely, if the
height of the diagram is restricted to $r$ (in our case say $r=3$), then $c_.(r)$ does not change. In this case
$$\mathfrak{S}(P)=\,\Yvcentermath1\young(\ku\ku+--,\ku,\ku)\, .$$
\end{example}

\subsection{Kirillov--Reshetikhin crystals of type $D_n^{(1)}$} \label{section.KR}

Let $\geh$ be an affine Kac--Moody Lie algebra with index set $I=\{0,1,\ldots,n\}$.
Kirillov--Reshetikhin (KR) crystals $B^{r,s}$ are indexed by $r\in \J := I \setminus \{0\}$ and an integer $s\ge 1$.
For nonexceptional types their existence was proven in~\cite{O:2007,OS:2008}. In this paper we only
deal with the KR crystals $B^{r,s}$ of type $D_n^{(1)}$.

As classical crystals the Kirillov--Reshetikhin crystals usually decompose into several components.
For type $D_n^{(1)}$, the classical decomposition of $B^{r,s}$ for $1\le r\le n-2$ is
\begin{equation} \label{equation.KR decomposition}
	B^{r,s} \cong \bigoplus_\lambda B(\lambda) \qquad \text{as $D_n$-crystals,}
\end{equation}
where the sum is over all partitions (or equivalently weights) obtained from $(s^r)$ by removing vertical dominoes.
Each term appears with multiplicity one. For the spin cases $r=n-1,n$ we have
\begin{equation} \label{equation.KR decomposition spin}
	B^{r,s} \cong B(s\varpi_r) \qquad \text{as $D_n$-crystals.}
\end{equation}

The affine Kashiwara crystal operators $f_0$ and $e_0$ are defined as
\begin{equation}
	f_0 = \sigma \circ f_1 \circ \sigma \qquad \text{and} \qquad e_0 = \sigma \circ e_1 \circ \sigma \;,
\end{equation}
where $\sigma$ is the analogue of the Dynkin automorphism which interchanges nodes 0 and 1 as given in the 
next definition.

\begin{definition} \label{definition.sigma}
Let $t\in B^{r,s}$ with $B^{r,s}$ a KR crystal of type $D_n^{(1)}$ with $1\le r\le n-2$.
Choose a sequence $\mathbf{b}=(b_1,b_2,\ldots,b_k)$ with $b_i\in\{2,3,\ldots,n\}$ such that 
$e_{\mathbf{b}}(t):=e_{b_1} e_{b_2} \cdots e_{b_k}(t)$ is $\{2,3,\ldots,n\}$-highest weight. Then define
\begin{equation}
	\sigma(t) := f_{\mathbf{b}'} \circ \gamma \circ \; \mathfrak{S} \circ \gamma^{-1} 
	\circ \;e_{\mathbf{b}}(t)\; ,
\end{equation}
where $\mathbf{b}'$ is the reverse of $\mathbf{b}$ so that $f_{\mathbf{b}'} := f_{b_k} \cdots f_{b_1}$,
$\mathfrak{S}$ as in Definition~\ref{definition.sigma pm} with the heights restricted to $1\le h\le r$,
and $\gamma$ as in Proposition~\ref{proposition.gamma}.
\end{definition}

Let $B_1$, $B_2$ be two affine crystals with generators $v_1$ and $v_2$, respectively, such
that $B_1\otimes B_2$ is connected and $v_1\otimes v_2$ lies in a one-dimensional weight space.
By~\cite[Proposition 3.8]{LOS:2011}, this holds for any two
KR crystals. The generator $v$ for the KR crystal $B^{r,s}$ is chosen to be the unique element of classical weight 
$s\varpi_r$.
 
The \textbf{combinatorial $R$-matrix}~\cite[Section 4]{KMN1:1992} is the unique affine
crystal isomorphism
\begin{equation} \label{equation.R}
    R : B_1 \otimes B_2 \to B_2 \otimes B_1.
\end{equation}
By weight considerations, this must satisfy $R(v_1 \otimes v_2) = v_2 \otimes v_1$.

On (tensor products of) Kirillov--Reshetikhin crystals, there is a \textbf{(co)energy} function defined
\begin{equation} \label{equation.energy}
	\Db : B \to \Z.
\end{equation}
The (co)energy is constant on classical components. For a single $B^{r,s}$ the coenergy for the classical component 
$B(\lambda)$ is equal to the number of vertical dominoes in $(s^r) \setminus \lambda$~\cite{HKOTY:1999} (see 
also~\cite[Definition 5.4, Theorem 7.5]{ST:2011}). For the definition of the (co)energy on general tensor factors, see for 
example~\cite[Theorem 2.4]{OSS:2003} \cite{ST:2011}.

\section{Rigged configurations}
\label{section.rc}

Let $\geh$ be a simply-laced affine Kac--Moody algebra with index set $I$ of the underlying Dynkin diagram.
Recall that $\J=I\setminus\{0\}$ is the index set of the underlying algebra of finite
type $\geh_0$ and set $\HH=\J \times \Z_{>0}$. The (highest-weight) rigged configurations
are indexed by a multiplicity array $L=(L_i^{(a)}\mid (a,i)\in \HH)$ of
nonnegative integers and a dominant weight $\La$ of $\geh_0$. Note that only finitely many $L_i^{(a)}$
in the multiplicity array $L$ are nonzero.
The sequence of partitions $\nu=\{\nu^{(a)}\mid a\in \J \}$ is an
\textbf{$(L,\La)$-configuration} if
\begin{equation}\label{eq:conf}
\sum_{(a,i)\in\HH} i m_i^{(a)} \alpha_a = \sum_{(a,i)\in\HH} i
L_i^{(a)} \varpi_a- \La,
\end{equation}
where $m_i^{(a)}$ is the number of parts of length $i$ in partition
$\nu^{(a)}$. Denote the set of all $(L,\La)$-configurations by $\Conf(L,\La)$.
The \textbf{vacancy number} $p_i^{(a)}$ of a configuration is defined as
\begin{equation}\label{equation.vacancy}
	p_i^{(a)}=\sum_{j\ge 1} \min(i,j) L_j^{(a)}
 	- \sum_{(b,j)\in \HH} (\alpha_a | \alpha_b) \min(i,j)m_j^{(b)}.
\end{equation}
Here $(\cdot | \cdot )$ is the normalized invariant form on the weight lattice $P$
such that $A_{ab}=(\alpha_a | \alpha_b)$ is the Cartan matrix.
The $(L,\La)$-configuration $\nu$ is \textbf{admissible} if $p^{(a)}_i\ge 0$ for 
all $(a,i)\in\HH$, and the set of admissible $(L,\La)$-configurations is denoted
by $\Confb(L,\La)$.

A \textbf{rigged configuration} is an admissible configuration together with
a set of labels of quantum numbers. A partition can be viewed as
a multiset of positive integers.
A rigged partition is by definition a finite multiset of
pairs $(i,x)$ where $i$ is a positive integer and
$x$ is a nonnegative integer.  The pairs $(i,x)$ are referred to
as \textbf{strings}; $i$ is referred to as the
length or size of the string and $x$ as the \textbf{label} or
\textbf{rigging} of the string.  A rigged partition is
said to be a rigging of the partition $\rho$ if
the multiset, consisting of the sizes of the strings,
is the partition $\rho$.  So a rigging of $\rho$
is a labeling of the parts of $\rho$ by nonnegative integers,
where one identifies labelings that differ only by
permuting labels among equal sized parts of $\rho$.

A rigging $J$ of the $(L,\La)$-configuration $\nu$ is a sequence of
riggings of the partitions $\nu^{(a)}$ such that
every label $x$ of a part of $\nu^{(a)}$ of size $i$
satisfies the inequality
\begin{equation*}
  0 \le x \le p^{(a)}_i.
\end{equation*}
Alternatively, a rigging of a configuration $\nu$ may be viewed
as a double-sequence of partitions $J=(J^{(a,i)}\mid (a,i)\in\HH)$
where $J^{(a,i)}$ is a partition that has at most $m_i^{(a)}$ parts
each not exceeding $p_i^{(a)}$.
The pair $(\nu,J)$ is called a rigged configuration.

\begin{definition}
The set of riggings of admissible $(L,\La)$-configurations is denoted by $\RCb(L,\La)$.
Let $(\nu,J)^{(a)} = (\nu^{(a)},J^{(a)})$ be the $a$-th rigged partition of $(\nu,J)$. 
The \textbf{colabel} or \textbf{corigging} of a string $(i,x)$ in $(\nu,J)^{(a)}$ is defined to be $p_i^{(a)}-x$.
A string $(i,x)\in (\nu,J)^{(a)}$ is said to be \textbf{singular} if $x=p^{(a)}_i$, that is,
its label takes on the maximum value.
We also set
\[
	\RCb(L) = \bigsqcup_{\Lambda \in P^+} \RCb(L,\Lambda),
\]
where $P^+$ is the set of dominant weights.
\end{definition}

\begin{remark}
Given a tensor product of KR crystals $B = B^{r_1,s_1} \otimes \cdots \otimes B^{r_k,s_k}$ with
$(r_i,s_i) \in \HH$, we can associate to it a multiplicity array $L = (L_s^{(r)} \mid (r,s) \in \HH)$, where
$L_s^{(r)}$ counts the number of tensor factors $B^{r,s}$ in $B$. Given this natural correspondence we sometimes
also use the notation $\RCb(B,\Lambda)$ or $\RCb(B)$. Note, however, that the rigged configurations do not
depend on the order of the tensor factors in $B$, just their multiplicities.
\end{remark}

The set of rigged configurations is endowed with a natural
statistic $\cc$ called \textbf{cocharge}. For a configuration 
$\nu\in\Confb(L,\La)$ define
\begin{equation} \label{equation.cocharge}
	\cc(\nu)=\frac{1}{2} \sum_{(a,j),(b,k)\in \HH} (\alpha_a|\alpha_b) 
 	\min(j,k) m_j^{(a)} m_k^{(b)}.
\end{equation}
For a rigged configuration $(\nu,J)\in\RCb(L,\La)$ set
\begin{equation} \label{equation.cocharge J}
  	\cc(\nu,J)=\cc(\nu)+\sum_{(a,i)\in\HH} |J^{(a,i)}|,
\end{equation}
where $|J^{(a,i)}|$ is the size of partition $J^{(a,i)}$.

In~\cite{Kleber:1998}, Kleber gave an algorithm to produce all admissible configurations for a given sequence of
rectangles. In particular his algorithm shows that for type $D_n^{(1)}$ and for a single tensor factor $B^{r,s}$
all admissible rigged configurations are given as follows:

\begin{proposition}[Kleber~\cite{Kleber:1998}]
\label{proposition.kleber}
Let $1\le r \le n-2$, $s\ge 1$ be integers, and $B^{r,s}$ a Kirillov--Reshetikhin crystal of type $D_n^{(1)}$. Then
\begin{equation} \label{equation.RC decomposition}
	\RCb(B^{r,s}) = \bigsqcup_\lambda \RCb(B^{r,s},\lambda)\;,
\end{equation}
where $\lambda$ is obtained from $(s^r)$ by removing vertical dominoes. In addition, $\RCb(B^{r,s},\lambda)$
contains the single element $(\nu,J)$ with
\begin{equation}
\nu^{(a)} = \begin{cases} 
	\overline{\lambda}^{[r-a]} & \text{for $1\le a < r$,}\\
	\overline{\lambda}             & \text{for $r\le a\le n-2$,}\\
	\overline{\lambda}'            & \text{for $a=n-1,n$,}
\end{cases}
\end{equation}
and all riggings in $J$ being zero. Here $\overline{\lambda}$ is the complement of the partition $\lambda$ in the rectangle 
$(s^r)$, $\overline{\lambda}^{[b]}$ is obtained from $\overline{\lambda}$ by removing the $b$ longest rows,
and $\overline{\lambda}'$ is obtained from $\overline{\lambda}$ considering only odd rows.
\end{proposition}

\begin{example}
Let $b$ be the highest weight element of $B^{6,4}$ of type $D^{(1)}_8$
with weight $\varpi_4+2\varpi_2$. As a KN tableau
$b=\Yvcentermath1\young(111,222,3,4)$
and its complement within the $6\times 4$ rectangle has shape
$\Yvcentermath1\yng(4,4,3,3,1,1)$.
Then the rigged configuration corresponding to $b$ is as follows:
\begin{center}
\unitlength 9pt
\begin{picture}(50,6.5)(0,-0.5)
\put(1,5){$\Yboxdim9pt\yng(1)$}
\put(2.2,5.1){0}
\put(4,4){$\Yboxdim9pt\yng(1,1)$}
\put(5.2,5.1){0}
\put(5.2,4.1){0}
\put(7,3.1){$\Yboxdim9pt\yng(3,1,1)$}
\put(10.2,5.1){0}
\put(8.2,4.1){0}
\put(8.2,3.1){0}
\put(12,2.1){$\Yboxdim9pt\yng(3,3,1,1)$}
\put(15.2,5.1){0}
\put(15.2,4.1){0}
\put(13.2,3.1){0}
\put(13.2,2.1){0}
\put(17,1.1){$\Yboxdim9pt\yng(4,3,3,1,1)$}
\put(21.2,5.1){0}
\put(20.2,4.1){0}
\put(20.2,3.1){0}
\put(18.2,2.1){0}
\put(18.2,1.1){0}
\put(23,0.1){$\Yboxdim9pt\yng(4,4,3,3,1,1)$}
\put(27.2,5.1){0}
\put(27.2,4.1){0}
\put(26.2,3.1){0}
\put(26.2,2.1){0}
\put(24.2,1.1){0}
\put(24.2,0.1){0}
\put(29,3){$\Yboxdim9pt\yng(4,3,1)$}
\put(33.2,5.1){0}
\put(32.2,4.1){0}
\put(30.2,3.1){0}
\put(35,3){$\Yboxdim9pt\yng(4,3,1)$}
\put(39.2,5.1){0}
\put(38.2,4.1){0}
\put(36.2,3.1){0}
\end{picture}
\end{center}
Here we express the configuration $\nu^{(a)}$ by its Young diagram and
put riggings on the right of the corresponding rows.
In particular, $\nu^{(6)}$ coincides with the complement of the shape of $b$.
To obtain the configurations to the left one removes the top row one by one.
\end{example}

We now review the fact that the set of rigged configurations is also endowed with a classical crystal structure.

\begin{definition} \rm{(\cite[Definition 3.3]{S:2006})} \label{definition.rc crystal}
Let $L$ be a multiplicity array.
Define the set of \textbf{unrestricted rigged configurations} $\RC(L)$ 
as the set generated from the elements in $\RCb(L)$ by the application of  
the operators $e_a,f_a$ for $a\in \J$ defined as follows:
\begin{enumerate}
\item
Define $e_a(\nu,J)$ by removing a box from a string of length $k$ in
$(\nu,J)^{(a)}$ leaving all colabels fixed and increasing the new
label by one. Here $k$ is the length of the string with the smallest
negative rigging of smallest length. If no such string exists,
$e_a(\nu,J)$ is undefined. 
\item
Define $f_a(\nu,J)$ by adding a box to a string of length $k$ in
$(\nu,J)^{(a)}$ leaving all colabels fixed and decreasing the new
label by one. Here $k$ is the length of the string with the smallest
nonpositive rigging of largest length. If no such string exists,
add a new string of length one and label $-1$.
If the result is not a valid unrestricted rigged configuration (meaning that
all riggings are smaller than or equal to their corresponding vacancy numbers),
$f_a(\nu,J)$ is undefined.
\end{enumerate}
\end{definition}

It was shown in~\cite[Theorem 3.7]{S:2006} that the operators $f_a$ and $e_a$ for $a\in \J$ define a classical
crystal structure on the set of rigged configurations.

\begin{theorem} \label{theorem.crystal isomorphism}
Let $B^{r,s}$ be a Kirillov--Reshetikhin crystal of type $D_n^{(1)}$.
Then there is a $D_n$-crystal isomorphism $\iota_0$ between $B^{r,s}$ and $\RC(B^{r,s})$
\begin{equation}
	\iota_0 : B^{r,s} \cong \RC(B^{r,s}).
\end{equation}
\end{theorem}

\begin{proof}
First assume that $1\le r \le n-2$. Comparing~\eqref{equation.KR decomposition} and Proposition~\ref{proposition.kleber}, 
there is a bijection between $D_n$-highest weight elements in $B^{r,s}$ and rigged configurations $\RCb(B^{r,s})$
which is unique since all weights have multiplicity one. By~\cite[Theorem 3.7]{S:2006} the corresponding 
classical crystal structures agree.

For $r=n-1,n$ we have $B^{r,s} \cong B(s\varpi_r)$ as $D_n$-crystals by~\eqref{equation.KR decomposition spin}.
The Kleber algorithm~\cite{Kleber:1998} shows that there is only the empty rigged configuration in $\RCb(B^{r,s})$.
Again, by~\cite[Theorem 3.7]{S:2006} this proves the claim.
\end{proof}

In the next section we show that the classical crystal isomorphism $\iota_0$ of Theorem~\ref{theorem.crystal isomorphism}
can in fact be extended to an affine crystal isomorphism.

\section{Affine crystal structure on rigged configurations}
\label{section.affine structure}

In this section we define an affine crystal structure on rigged configurations. This is achieved by using the  classical 
crystal structure of Definition~\ref{definition.rc crystal} and defining the analogue of $\sigma$ of Definition~\ref{definition.sigma}
on rigged configurations. Our main result is stated in Theorem~\ref{theorem.main}.

\subsection{$\pm$-diagrams on rigged configurations}
\label{section.pm for rc}

In this subsection we define rigged configurations associated to $\pm$-diagrams and show in Proposition~\ref{proposition.rc pm}
that they indeed correspond to $\{2,3,\ldots,n\}$-highest weight crystal elements.
Here we only consider $B^{r,s}$ of type $D_n^{(1)}$ with $1\le r\le n-2$.

For a given $\pm$-diagram of type $D_n$, the corresponding rigged configuration
is obtained by ``adding'' all the rigged configurations corresponding to the 
single columns of the $\pm$-diagram together.
Here ``adding" means concatenating the parts of all Young diagrams of the rigged configuration horizontally and 
summing up the corresponding riggings.

Hence in order to obtain the rigged configuration for a 
$\pm$-diagram of $B^{r,s}$, it is enough to know the following information.
Let $P$ be a single column $\pm$-diagram of height $x$ of the outer shape corresponding to a
$\{2,3,\ldots,n\}$-highest weight element in $B^{x+y,1}$, where $x+y=r$.
We remark that $y$ is always an even integer by the classical decomposition~\eqref{equation.KR decomposition}.
In order to display the rigged configuration, we represent a Young diagram by the sequence of lengths of rows like $(1^i)$ 
and if $i=0$ we regard them as the empty set.
We describe the riggings just below the corresponding rows.

\begin{enumerate}
\item[(A)] $P$ does not contain any sign.
\begin{align*}
\nu &=(\overbrace{(\hspace{2.5mm}1), (1),\cdots,(1)}^x,
\overbrace{(1),(1,1),\cdots,(1,\ldots,1)}^y,
(1^y),\cdots,(1^y),(1^{\frac{y}{2}}),(1^{\frac{y}{2}}))\\
J &=(\overbrace{(-1), (0),\cdots,(0)}^x,
\overbrace{(1),(0,0),\cdots,(0,\ldots,0)}^y,
(0^y),\cdots,(0^y),(0^{\frac{y}{2}}),(0^{\frac{y}{2}}))\\
\end{align*}

\item[(B)] $P$ contains $+$.
\begin{align*}
\nu &=(\overbrace{\emptyset, \emptyset,\cdots,\emptyset}^x,
\overbrace{(1),(1,1),\cdots,(1,\ldots,1)}^y,
(1^y),\cdots,(1^y),(1^{\frac{y}{2}}),(1^{\frac{y}{2}}))\\
J &=(\overbrace{\emptyset, \emptyset,\cdots,\emptyset}^x,
\overbrace{(0),(0,0),\cdots,(0,\ldots,0)}^y,
(0^y),\cdots,(0^y),(0^{\frac{y}{2}}),(0^{\frac{y}{2}}))\\
\end{align*}

\item[(C)] $P$ contains $-$.
\begin{align*}
\nu &=(\overbrace{(\hspace{2.5mm}2), (2),\cdots,(2)}^{x-1},
\overbrace{(1,1),(1,1,1),\cdots,(1,\ldots,1)}^{y+1},
(1^{y+2}),\cdots,(1^{y+2}),(1^{\frac{y+2}{2}}),(1^{\frac{y+2}{2}}))\\
J &=(\overbrace{(-2), (0),\cdots,(0)}^{x-1},
\overbrace{(0,0),(0,0,0),\cdots,(0,\ldots,0)}^{y+1},
(0^{y+2}),\cdots,(0^{y+2}),(0^{\frac{y+2}{2}}),(0^{\frac{y+2}{2}}))\\
\end{align*}

\item[(D)] $P$ contains $\pm$.
\begin{align*}
\nu &=(\overbrace{(\hspace{2.5mm}1), (1),\cdots,(1)}^{x-1},
\overbrace{(1,1),(1,1,1),\cdots,(1,\ldots,1)}^{y+1},
(1^{y+2}),\cdots,(1^{y+2}),(1^{\frac{y+2}{2}}),(1^{\frac{y+2}{2}}))\\
J &=(\overbrace{(-1), (0),\cdots,(0)}^{x-1},
\overbrace{(0,0),(0,0,0),\cdots,(0,\ldots,0)}^{y+1},
(0^{y+2}),\cdots,(0^{y+2}),(0^{\frac{y+2}{2}}),(0^{\frac{y+2}{2}}))\\
\end{align*}
\end{enumerate}
\noindent
Except for $\nu^{(x-1)}$ for the $\pm$ case, all rows are singular.
The empty $\pm$-diagram is regarded as a special case of the $+$ case.
If we have $x=1$ in the $-$ case, we take $(\nu^{(1)},J^{(1)})=((1,1),(-1,-1))$.

\begin{definition}\label{definition.pm rc}
Let us denote the rigged configuration obtained from a $\pm$-diagram $P$  for $B^{r,s}$ by the above
procedure by $\gammarc(P):=(\nu_P,J_P)$.
\end{definition}

\begin{example}
Consider the following element of $B^{8,5}$ of type $D^{(1)}_n$ $(n\geq 10)$:
$$P=\Yvcentermath1\young(\ku\ku\ku+,\ku\ku\ku-,\ku\ku\ku,\ku\ku\ku,\ku+,--)$$
Then the corresponding RC is as follows (the first line is $\nu_P$
and the second line is $J_P$):
\begin{align*}
((\hspace{2.5mm}6),(6,2),(6,2,2),(6,2,2,2),(6,2,2,2,2),(5,5,2,2,2,2),
(5,5,5,2,2,2,2),\cdots)\\
((-5),(0,0),(0,0,0),(0,0,0,0),(1,0,0,0,0),(0,0,0,0,0,0),
(0,0,0,0,0,0,0),\cdots)
\end{align*}
Take $\nu_P^{(5)}$ as an example. It is the concatenation (or ``sum'') of the five columns of $P$:
$$\Yvcentermath1\yng(2)\,\,0,\,\,
\yng(1)\,\,0,\,\,\yng(1)\,\,1,\,\,\yng(1,1,1,1,1)
\unitlength 11pt
\begin{picture}(1,2)
\multiput(0.3,-2.2)(0,1.05){5}{0}
\end{picture},\,\,\yng(1,1,1,1,1)
\unitlength 11pt
\begin{picture}(1,2)
\multiput(0.3,-2.2)(0,1.05){5}{0}
\end{picture}.$$
Here we put riggings on the right of the corresponding rows.
Adding these together we obtain
\begin{center}
\unitlength 11pt
\begin{picture}(7,5.3)
\put(0,0){\yng(6,2,2,2,2)}
\put(6.4,4.2){1}
\multiput(2.4,0.1)(0,1.05){4}{0}
\end{picture}.
\end{center}
\end{example}

\begin{proposition} \label{proposition.rc pm}
We have
\[
	\gammarc = \iota_0 \circ \gamma.
\]
\end{proposition}

\begin{proof}
To prove the claim we show that the combinatorially defined map $\gammarc$ on rigged configurations of 
Definition~\ref{definition.pm rc} follows the same inductive definition as $\gamma$ as given in
Proposition~\ref{proposition.gamma}. In fact we prove the equivalent property that 
\[
	\gammarc(P') = \begin{cases}
	e_h e_{h-1} \cdots e_1 \gammarc(P) & \text{for Case 1 of Proposition~\ref{proposition.gamma},}\\
	e_h e_{h+1} \cdots e_{n-2} e_n e_{n-1} \cdots e_1 \gammarc(P) & \text{for Case 2 of Proposition~\ref{proposition.gamma},}
	\end{cases}
\]
where $h$ is the height of the added $+$ in Case 1 and the removed $-$ in Case 2.

\noindent \textbf{Case 1:}
There are two cases. Let $c$ be the rightmost column, where a $+$ can be added.
\begin{itemize}
\item[(a)] $c$ does not contain signs.
\item[(b)] $c$ contains only $-$.
\end{itemize}
Note that for (a) (resp. (b)) the height of $c$ is $h$ (resp. $h+1$).

We first treat case (a). From the rules in Section~\ref{section.pm for rc} the map $\gammarc(P)$ has the
following features:
\begin{itemize}
\item[(i)] $\nu^{(1)}=(N),J^{(1)}=(-N)$ for some positive integer $N$.
\item[(ii)] $J^{(a)}_j=0$ for $2\le a\le h$ and $j\ge1$.
\item[(iii)] $\nu^{(a)}_1\le\nu^{(a+1)}_1$ for $1\le a\le h$.
\item[(iv)] $\nu^{(a)}_1>\nu^{(a+1)}_j$ for $1\le a\le h$ and $j\ge2$.
\item[(v)] $J^{(h+1)}_j=m\cdot \chi(j=1)$ for any $j$, where $m$ is the number of columns without sign of height $h$.
\end{itemize}
In the above, $J^{(a)}_i$ stands for the rigging of the $i$-th row in $(\nu,J)^{(a)}$ and 
\begin{equation} \label{equation.chi}
	\chi(S) = \begin{cases} 1 & \text{if the statement $S$ is true,}\\
		0 & \text{if the statement $S$ is false.}
	\end{cases}
\end{equation}
{}From Definition~\ref{definition.rc crystal}, $e_1\gammarc(P)$ differs from $\gammarc(P)$ by removing a box
from the first row of $\nu^{(1)}$, increasing $J^{(1)}_1$ by one and decreasing
$J^{(2)}_1$ by one. Applying $e_2,\ldots,e_h$ similarly, one recognizes that
$e_h\cdots e_1\gammarc(P)$ differs from $\gammarc(P)$ by removing a
box from the first row of $\nu^{(a)}$ for $1\le a\le h$, increasing $J^{(1)}_1$ by
one and decreasing $J^{(h+1)}_1$ by one. This is exactly the difference between Case (A) and (B)
in Section~\ref{section.pm for rc} in which $\gammarc(P)$ and $\gammarc(P')$ differ.

The proof of case (b) goes similarly. The features of $\gammarc(P)$ differ
from case (a) in
\begin{itemize}
\item[(iii')] $\nu^{(a)}_1\le\nu^{(a+1)}_1$ for $1\le a\le h-1$, and $\nu^{(h)}_1=\nu^{(h+1)}_1+m$, where $m$ is the number
of columns with only $-$ of height $h+1$.
\item[(v')] $J^{(h+1)}_j=0$ for all $j$.
\end{itemize}
In this case the difference between $\gammarc(P)$ and $\gammarc(P')$ is exactly the difference between Case (C)
and (D) in Section~\ref{section.pm for rc}.

\noindent \textbf{Case 2:} Since no $+$ can be added to $P$, every column either contains a $+$ or is of height 1 and contains
a $-$. Hence we only encounter Cases (B), (D), or Case (C) with $x=1$ of Section~\ref{section.pm for rc}. The constructed 
rigged configuration $(\nu, J) = \gammarc(P)$ has the following property:
\begin{enumerate}
\item[(i)] $\nu^{(1)} = (N,L)$, $J^{(1)} = (-N,-L)$, where $N$ is the total number of $-$ in $P$ and $L$ is the number of $-$ at
height one.
\item[(ii)] $J^{(a)}_j = 0$ for $a\ge 2$ and all $j$.
\item[(iii)] $\nu_1^{(a)} \le \nu_1^{(a+1)}$ for $1\le a<h$.
\item[(iv)] $\nu_1^{(a)}>\nu_j^{(a+1)}$ for $1\le a<h-1$, $j\ge 2$ and $\nu_1^{(h)}>\nu_j^{(h+1)}$ for $j\ge 3$.
\item[(v)] $\nu^{(a)}$ for $h\le a\le n-2$ has at least two parts of length $M:=\nu_1^{(h-1)}$.
\item[(vi)] $\nu^{(n-1)}$ and $\nu^{(n)}$ have at least one part of length $M$.
\end{enumerate}
As before, conditions (i) to (iv) ensure that $e_h \cdots e_1$ always remove a box from the largest part in $\nu^{(a)}$ for
$1\le a\le h$. The next $e_n \cdots e_{h+1}$ remove a box from the parts of length $M$ in $\nu^{(a)}$ for $h<a\le n$.
Since by condition (v) there are at least two parts of length $M$ in $\nu^{(a)}$ for $h\le a\le n-2$, the following
$e_h e_{h+1} \cdots e_{n-2}$ pick the second part of length $M$. Then it is not hard to check that Case (D)
(or Case (C) with $x=1$) of Section~\ref{section.pm for rc} turns into Case (B), which is precisely the difference between
$\gammarc(P)$ and $\gammarc(P')$.
\end{proof}

\begin{example} \label{example.pm}
Take $\pm$-diagrams $P,P'$ corresponding to $\{2,3,4,5,6\}$-highest weight elements in $B^{4,3}$ of type $D_6^{(1)}$ 
as follows
$$P=\,\Yvcentermath1\young(\ku\ku+,\ku+-,\ku,-)\, ,\qquad
P'=\,\Yvcentermath1\young(\ku\ku+,\ku+-,+,-)\, .$$
Then we are in Case 1 of Proposition~\ref{proposition.gamma} since $P$ is obtained from $P'$ by removing a $+$
in the first column at height 3. Hence $\gammarc (P')=e_3e_2e_1\gammarc (P)$ which we can compute explicitly as follows:
\begin{center}
\unitlength 10pt
\begin{picture}(37,4)(0.5,0)
\put(0.5,3.1){$-2$}
\put(2.0,3){$\Yboxdim10pt\yng(3)$}
\put(5.1,3.1){$-3$}
\put(7,3.1){0}
\put(7,2.1){0}
\put(7.9,2){$\Yboxdim10pt\yng(3,1)$}
\put(9.2,2.1){0}
\put(11.2,3.1){0}
\put(13,3.1){0}
\put(13,2.1){0}
\put(13,1.1){0}
\put(13.9,1){$\Yboxdim10pt\yng(4,1,1)$}
\put(15.2,1.1){0}
\put(15.2,2.1){0}
\put(18.2,3.1){0}
\put(20,0.1){0}
\put(20,1.1){0}
\put(20,2.1){0}
\put(20,3.1){0}
\put(20.9,0){$\Yboxdim10pt\yng(3,3,1,1)$}
\put(22.2,0.1){0}
\put(22.2,1.1){0}
\put(24.2,2.1){0}
\put(24.2,3.1){0}
\put(26,2.1){0}
\put(26,3.1){0}
\put(26.9,2){$\Yboxdim10pt\yng(3,1)$}
\put(28.2,2.1){0}
\put(30.2,3.1){0}
\put(32,2.1){0}
\put(32,3.1){0}
\put(32.9,2){$\Yboxdim10pt\yng(3,1)$}
\put(34.2,2.1){0}
\put(36.2,3.1){0}
\end{picture}
\end{center}
\vspace{-3mm}

\begin{center}
\unitlength 10pt
\begin{picture}(1,2)
\put(-0.2,1.0){$e_1$}
\put(1,2){\vector(0,-1){1.8}}
\end{picture}
\end{center}
\vspace{-7mm}

\begin{center}
\unitlength 10pt
\begin{picture}(37,4)(0.5,0)
\put(0.5,3.1){$-1$}
\put(2.0,3){$\Yboxdim10pt\yng(2)$}
\put(4.1,3.1){$-2$}
\put(6.3,3.1){$-1$}
\put(7,2.1){0}
\put(7.9,2){$\Yboxdim10pt\yng(3,1)$}
\put(9.2,2.1){0}
\put(11.0,3.1){$-1$}
\put(13,3.1){0}
\put(13,2.1){0}
\put(13,1.1){0}
\put(13.9,1){$\Yboxdim10pt\yng(4,1,1)$}
\put(15.2,1.1){0}
\put(15.2,2.1){0}
\put(18.2,3.1){0}
\put(20,0.1){0}
\put(20,1.1){0}
\put(20,2.1){0}
\put(20,3.1){0}
\put(20.9,0){$\Yboxdim10pt\yng(3,3,1,1)$}
\put(22.2,0.1){0}
\put(22.2,1.1){0}
\put(24.2,2.1){0}
\put(24.2,3.1){0}
\put(26,2.1){0}
\put(26,3.1){0}
\put(26.9,2){$\Yboxdim10pt\yng(3,1)$}
\put(28.2,2.1){0}
\put(30.2,3.1){0}
\put(32,2.1){0}
\put(32,3.1){0}
\put(32.9,2){$\Yboxdim10pt\yng(3,1)$}
\put(34.2,2.1){0}
\put(36.2,3.1){0}
\end{picture}
\end{center}
\vspace{-3mm}

\begin{center}
\unitlength 10pt
\begin{picture}(1,2)
\put(-0.2,1.0){$e_2$}
\put(1,2){\vector(0,-1){1.8}}
\end{picture}
\end{center}
\vspace{-7mm}

\begin{center}
\unitlength 10pt
\begin{picture}(37,4)(0.5,0)
\put(0.5,3.1){$-1$}
\put(2.0,3){$\Yboxdim10pt\yng(2)$}
\put(4.1,3.1){$-2$}
\put(7,3.1){0}
\put(7,2.1){0}
\put(7.9,2){$\Yboxdim10pt\yng(2,1)$}
\put(9.2,2.1){0}
\put(10.2,3.1){0}
\put(12.3,3.1){$-1$}
\put(13,2.1){0}
\put(13,1.1){0}
\put(13.9,1){$\Yboxdim10pt\yng(4,1,1)$}
\put(15.2,1.1){0}
\put(15.2,2.1){0}
\put(18.0,3.1){$-1$}
\put(20,0.1){0}
\put(20,1.1){0}
\put(20,2.1){0}
\put(20,3.1){0}
\put(20.9,0){$\Yboxdim10pt\yng(3,3,1,1)$}
\put(22.2,0.1){0}
\put(22.2,1.1){0}
\put(24.2,2.1){0}
\put(24.2,3.1){0}
\put(26,2.1){0}
\put(26,3.1){0}
\put(26.9,2){$\Yboxdim10pt\yng(3,1)$}
\put(28.2,2.1){0}
\put(30.2,3.1){0}
\put(32,2.1){0}
\put(32,3.1){0}
\put(32.9,2){$\Yboxdim10pt\yng(3,1)$}
\put(34.2,2.1){0}
\put(36.2,3.1){0}
\end{picture}
\end{center}
\vspace{-3mm}

\begin{center}
\unitlength 10pt
\begin{picture}(1,2)
\put(-0.2,1.0){$e_3$}
\put(1,2){\vector(0,-1){1.8}}
\end{picture}
\end{center}
\vspace{-7mm}

\begin{center}
\unitlength 10pt
\begin{picture}(37,4)(0.5,0)
\put(0.5,3.1){$-1$}
\put(2.0,3){$\Yboxdim10pt\yng(2)$}
\put(4.1,3.1){$-2$}
\put(7,3.1){0}
\put(7,2.1){0}
\put(7.9,2){$\Yboxdim10pt\yng(2,1)$}
\put(9.2,2.1){0}
\put(10.2,3.1){0}
\put(13,3.1){1}
\put(13,2.1){0}
\put(13,1.1){0}
\put(13.9,1){$\Yboxdim10pt\yng(3,1,1)$}
\put(15.2,1.1){0}
\put(15.2,2.1){0}
\put(17.0,3.1){0}
\put(20,0.1){0}
\put(20,1.1){0}
\put(20,2.1){0}
\put(20,3.1){0}
\put(20.9,0){$\Yboxdim10pt\yng(3,3,1,1)$}
\put(22.2,0.1){0}
\put(22.2,1.1){0}
\put(24.2,2.1){0}
\put(24.2,3.1){0}
\put(26,2.1){0}
\put(26,3.1){0}
\put(26.9,2){$\Yboxdim10pt\yng(3,1)$}
\put(28.2,2.1){0}
\put(30.2,3.1){0}
\put(32,2.1){0}
\put(32,3.1){0}
\put(32.9,2){$\Yboxdim10pt\yng(3,1)$}
\put(34.2,2.1){0}
\put(36.2,3.1){0}
\end{picture}
\end{center}
Here we put riggings (resp. vacancy numbers) on the right (resp. left)
of the corresponding rows.
\end{example}

\begin{example}
Continuing Example~\ref{example.pm} we take in addition
$$P''=\,\Yvcentermath1\young(\ku\ku+,\ku+-,\ku,+)\, .$$
Then we are in Case 2 of Proposition~\ref{proposition.gamma} since $P''$ is obtained from $P'$ by removing a $-$ in the
first column at height 4. Hence $\gammarc (P'')=e_4e_6e_5e_4e_3e_2e_1\gammarc (P')$ which can be checked explicitly
as follows:
\begin{center}
\unitlength 10pt
\begin{picture}(34,4)(0.5,0)
\put(0.5,3.1){$-1$}
\put(2.0,3){$\Yboxdim10pt\yng(2)$}
\put(4.1,3.1){$-2$}
\put(6,3.1){0}
\put(6,2.1){0}
\put(6.9,2){$\Yboxdim10pt\yng(2,1)$}
\put(8.2,2.1){0}
\put(9.2,3.1){0}
\put(11,3.1){1}
\put(11,2.1){0}
\put(11,1.1){0}
\put(11.9,1){$\Yboxdim10pt\yng(3,1,1)$}
\put(13.2,1.1){0}
\put(13.2,2.1){0}
\put(15.2,3.1){0}
\put(17,0.1){0}
\put(17,1.1){0}
\put(17,2.1){0}
\put(17,3.1){0}
\put(17.9,0){$\Yboxdim10pt\yng(3,3,1,1)$}
\put(19.2,0.1){0}
\put(19.2,1.1){0}
\put(21.2,2.1){0}
\put(21.2,3.1){0}
\put(23,2.1){0}
\put(23,3.1){0}
\put(23.9,2){$\Yboxdim10pt\yng(3,1)$}
\put(25.2,2.1){0}
\put(27.2,3.1){0}
\put(29,2.1){0}
\put(29,3.1){0}
\put(29.9,2){$\Yboxdim10pt\yng(3,1)$}
\put(31.2,2.1){0}
\put(33.2,3.1){0}
\end{picture}
\end{center}

\begin{center}
\unitlength 10pt
\begin{picture}(1,2)
\put(-0.2,1.0){$e_1$}
\put(1,2){\vector(0,-1){2}}
\end{picture}
\end{center}

\begin{center}
\unitlength 10pt
\begin{picture}(34,4)(0.5,0)
\put(1.1,3.1){$0$}
\put(2.0,3){$\Yboxdim10pt\yng(1)$}
\put(3.1,3.1){$-1$}
\put(5.3,3.1){$-1$}
\put(6,2.1){0}
\put(6.9,2){$\Yboxdim10pt\yng(2,1)$}
\put(8.2,2.1){0}
\put(9.1,3.1){$-1$}
\put(11,3.1){1}
\put(11,2.1){0}
\put(11,1.1){0}
\put(11.9,1){$\Yboxdim10pt\yng(3,1,1)$}
\put(13.2,1.1){0}
\put(13.2,2.1){0}
\put(15.2,3.1){0}
\put(17,0.1){0}
\put(17,1.1){0}
\put(17,2.1){0}
\put(17,3.1){0}
\put(17.9,0){$\Yboxdim10pt\yng(3,3,1,1)$}
\put(19.2,0.1){0}
\put(19.2,1.1){0}
\put(21.2,2.1){0}
\put(21.2,3.1){0}
\put(23,2.1){0}
\put(23,3.1){0}
\put(23.9,2){$\Yboxdim10pt\yng(3,1)$}
\put(25.2,2.1){0}
\put(27.2,3.1){0}
\put(29,2.1){0}
\put(29,3.1){0}
\put(29.9,2){$\Yboxdim10pt\yng(3,1)$}
\put(31.2,2.1){0}
\put(33.2,3.1){0}
\end{picture}
\end{center}

\begin{center}
\unitlength 10pt
\begin{picture}(1,2)
\put(-0.2,1.0){$e_2$}
\put(1,2){\vector(0,-1){2}}
\end{picture}
\end{center}

\begin{center}
\unitlength 10pt
\begin{picture}(34,4)(0.5,0)
\put(1.1,3.1){$0$}
\put(2.0,3){$\Yboxdim10pt\yng(1)$}
\put(3.1,3.1){$-1$}
\put(6.0,3.1){0}
\put(6.0,2.1){0}
\put(6.9,2){$\Yboxdim10pt\yng(1,1)$}
\put(8.2,2.1){0}
\put(8.2,3.1){0}
\put(11,3.1){0}
\put(11,2.1){0}
\put(11,1.1){0}
\put(11.9,1){$\Yboxdim10pt\yng(3,1,1)$}
\put(13.2,1.1){0}
\put(13.2,2.1){0}
\put(15.0,3.1){$-1$}
\put(17,0.1){0}
\put(17,1.1){0}
\put(17,2.1){0}
\put(17,3.1){0}
\put(17.9,0){$\Yboxdim10pt\yng(3,3,1,1)$}
\put(19.2,0.1){0}
\put(19.2,1.1){0}
\put(21.2,2.1){0}
\put(21.2,3.1){0}
\put(23,2.1){0}
\put(23,3.1){0}
\put(23.9,2){$\Yboxdim10pt\yng(3,1)$}
\put(25.2,2.1){0}
\put(27.2,3.1){0}
\put(29,2.1){0}
\put(29,3.1){0}
\put(29.9,2){$\Yboxdim10pt\yng(3,1)$}
\put(31.2,2.1){0}
\put(33.2,3.1){0}
\end{picture}
\end{center}

\begin{center}
\unitlength 10pt
\begin{picture}(1,2)
\put(-0.2,1.0){$e_3$}
\put(1,2){\vector(0,-1){2}}
\end{picture}
\end{center}

\begin{center}
\unitlength 10pt
\begin{picture}(34,4)(0.5,0)
\put(1.1,3.1){$0$}
\put(2.0,3){$\Yboxdim10pt\yng(1)$}
\put(3.1,3.1){$-1$}
\put(6.0,3.1){0}
\put(6.0,2.1){0}
\put(6.9,2){$\Yboxdim10pt\yng(1,1)$}
\put(8.2,2.1){0}
\put(8.2,3.1){0}
\put(11,3.1){0}
\put(11,2.1){0}
\put(11,1.1){0}
\put(11.9,1){$\Yboxdim10pt\yng(2,1,1)$}
\put(13.2,1.1){0}
\put(13.2,2.1){0}
\put(14.2,3.1){0}
\put(17,0.1){0}
\put(17,1.1){0}
\put(16.2,2.1){$-1$}
\put(16.2,3.1){$-1$}
\put(17.9,0){$\Yboxdim10pt\yng(3,3,1,1)$}
\put(19.2,0.1){0}
\put(19.2,1.1){0}
\put(21.0,2.1){$-1$}
\put(21.0,3.1){$-1$}
\put(23,2.1){0}
\put(23,3.1){0}
\put(23.9,2){$\Yboxdim10pt\yng(3,1)$}
\put(25.2,2.1){0}
\put(27.2,3.1){0}
\put(29,2.1){0}
\put(29,3.1){0}
\put(29.9,2){$\Yboxdim10pt\yng(3,1)$}
\put(31.2,2.1){0}
\put(33.2,3.1){0}
\end{picture}
\end{center}

\begin{center}
\unitlength 10pt
\begin{picture}(1,2)
\put(-0.2,1.0){$e_4$}
\put(1,2){\vector(0,-1){2}}
\end{picture}
\end{center}

\begin{center}
\unitlength 10pt
\begin{picture}(34,4)(0.5,0)
\put(1.1,3.1){$0$}
\put(2.0,3){$\Yboxdim10pt\yng(1)$}
\put(3.1,3.1){$-1$}
\put(6.0,3.1){0}
\put(6.0,2.1){0}
\put(6.9,2){$\Yboxdim10pt\yng(1,1)$}
\put(8.2,2.1){0}
\put(8.2,3.1){0}
\put(11,3.1){0}
\put(11,2.1){0}
\put(11,1.1){0}
\put(11.9,1){$\Yboxdim10pt\yng(2,1,1)$}
\put(13.2,1.1){0}
\put(13.2,2.1){0}
\put(14.2,3.1){0}
\put(17,0.1){0}
\put(17,1.1){0}
\put(17,2.1){0}
\put(17,3.1){1}
\put(17.9,0){$\Yboxdim10pt\yng(3,2,1,1)$}
\put(19.2,0.1){0}
\put(19.2,1.1){0}
\put(20.2,2.1){0}
\put(21.2,3.1){1}
\put(23,2.1){0}
\put(22.2,3.1){$-1$}
\put(23.9,2){$\Yboxdim10pt\yng(3,1)$}
\put(25.2,2.1){0}
\put(26.9,3.1){$-1$}
\put(29.2,2.1){0}
\put(28.4,3.1){$-1$}
\put(29.9,2){$\Yboxdim10pt\yng(3,1)$}
\put(31.2,2.1){0}
\put(33.2,3.1){$-1$}
\end{picture}
\end{center}

\begin{center}
\unitlength 10pt
\begin{picture}(1,2)
\put(-0.2,1.0){$e_5$}
\put(1,2){\vector(0,-1){2}}
\end{picture}
\end{center}

\begin{center}
\unitlength 10pt
\begin{picture}(34,4)(0.5,0)
\put(1.1,3.1){$0$}
\put(2.0,3){$\Yboxdim10pt\yng(1)$}
\put(3.1,3.1){$-1$}
\put(6.0,3.1){0}
\put(6.0,2.1){0}
\put(6.9,2){$\Yboxdim10pt\yng(1,1)$}
\put(8.2,2.1){0}
\put(8.2,3.1){0}
\put(11,3.1){0}
\put(11,2.1){0}
\put(11,1.1){0}
\put(11.9,1){$\Yboxdim10pt\yng(2,1,1)$}
\put(13.2,1.1){0}
\put(13.2,2.1){0}
\put(14.2,3.1){0}
\put(17,0.1){0}
\put(17,1.1){0}
\put(17,2.1){0}
\put(17,3.1){0}
\put(17.9,0){$\Yboxdim10pt\yng(3,2,1,1)$}
\put(19.2,0.1){0}
\put(19.2,1.1){0}
\put(20.2,2.1){0}
\put(21.2,3.1){0}
\put(23,2.1){0}
\put(23.0,3.1){0}
\put(23.9,2){$\Yboxdim10pt\yng(2,1)$}
\put(25.2,2.1){0}
\put(26.2,3.1){0}
\put(29.2,2.1){0}
\put(28.4,3.1){$-1$}
\put(29.9,2){$\Yboxdim10pt\yng(3,1)$}
\put(31.2,2.1){0}
\put(33.2,3.1){$-1$}
\end{picture}
\end{center}

\begin{center}
\unitlength 10pt
\begin{picture}(1,2)
\put(-0.2,1.0){$e_6$}
\put(1,2){\vector(0,-1){2}}
\end{picture}
\end{center}

\begin{center}
\unitlength 10pt
\begin{picture}(34,4)(0.5,0)
\put(1.1,3.1){$0$}
\put(2.0,3){$\Yboxdim10pt\yng(1)$}
\put(3.1,3.1){$-1$}
\put(6.0,3.1){0}
\put(6.0,2.1){0}
\put(6.9,2){$\Yboxdim10pt\yng(1,1)$}
\put(8.2,2.1){0}
\put(8.2,3.1){0}
\put(11,3.1){0}
\put(11,2.1){0}
\put(11,1.1){0}
\put(11.9,1){$\Yboxdim10pt\yng(2,1,1)$}
\put(13.2,1.1){0}
\put(13.2,2.1){0}
\put(14.2,3.1){0}
\put(17,0.1){0}
\put(17,1.1){0}
\put(17,2.1){0}
\put(16.2,3.1){$-1$}
\put(17.9,0){$\Yboxdim10pt\yng(3,2,1,1)$}
\put(19.2,0.1){0}
\put(19.2,1.1){0}
\put(20.2,2.1){0}
\put(21.0,3.1){$-1$}
\put(23,2.1){0}
\put(23.0,3.1){0}
\put(23.9,2){$\Yboxdim10pt\yng(2,1)$}
\put(25.2,2.1){0}
\put(26.2,3.1){0}
\put(29.2,2.1){0}
\put(29.2,3.1){0}
\put(29.9,2){$\Yboxdim10pt\yng(2,1)$}
\put(31.2,2.1){0}
\put(32.2,3.1){0}
\end{picture}
\end{center}

\begin{center}
\unitlength 10pt
\begin{picture}(1,2)
\put(-0.2,1.0){$e_4$}
\put(1,2){\vector(0,-1){2}}
\end{picture}
\end{center}

\begin{center}
\unitlength 10pt
\begin{picture}(34,4)(0.5,0)
\put(1.1,3.1){$0$}
\put(2.0,3){$\Yboxdim10pt\yng(1)$}
\put(3.1,3.1){$-1$}
\put(6.0,3.1){0}
\put(6.0,2.1){0}
\put(6.9,2){$\Yboxdim10pt\yng(1,1)$}
\put(8.2,2.1){0}
\put(8.2,3.1){0}
\put(11,3.1){0}
\put(11,2.1){0}
\put(11,1.1){0}
\put(11.9,1){$\Yboxdim10pt\yng(2,1,1)$}
\put(13.2,1.1){0}
\put(13.2,2.1){0}
\put(14.2,3.1){0}
\put(17,0.1){0}
\put(17,1.1){0}
\put(17,2.1){0}
\put(17,3.1){0}
\put(17.9,0){$\Yboxdim10pt\yng(2,2,1,1)$}
\put(19.2,0.1){0}
\put(19.2,1.1){0}
\put(20.2,2.1){0}
\put(20.2,3.1){0}
\put(23,2.1){0}
\put(23.0,3.1){0}
\put(23.9,2){$\Yboxdim10pt\yng(2,1)$}
\put(25.2,2.1){0}
\put(26.2,3.1){0}
\put(29.2,2.1){0}
\put(29.2,3.1){0}
\put(29.9,2){$\Yboxdim10pt\yng(2,1)$}
\put(31.2,2.1){0}
\put(32.2,3.1){0}
\end{picture}
\end{center}
\end{example}

Proposition \ref{proposition.rc pm} asserts that $\gammarc$ is a bijection.
Suppose a rigged configuration $(\nu,J)$ is given which is $\{2,\ldots,n\}$-highest weight. We give an algorithm 
to obtain the $\pm$-diagram $P=\gammarc^{-1}(\nu,J)$.
Recall from Section~\ref{section.sigma} that for $0\le h\le r$ the symbols $c_.(h),c_+(h),c_-(h),c_\pm(h)$ denote
the number of columns in $P$ of outer height $h$ with no sign, $+$, $-$, $\pm$, respectively.
Recall that for $r$ even (resp. odd) only even (resp. odd) values for $h$ exist.
The variables $c_a(h)$ for $a=\cdot,+,-,\pm$ are calculated inductively from $h=0$ ($r$ even) or $h=1$ ($r$ odd) 
to $h=r$ as follows:
\begin{align*}
c_.(h)& = \begin{cases}
	J^{(h+1)}_1+\chi(h=0)\nu^{(1)}_1 &(0\le h< r)\\
	\nu^{(r)}_1-\nu^{(r+1)}_1&(h=r)
	\end{cases}\\
c_+(h)& = \nu^{(h+1)}_1-\nu^{(h)}_1\qquad(1\le h<r)\\ 
c_-(h)&= \begin{cases}
	\nu^{(1)}_2&(h=1)\\ 
	\nu^{(h-1)}_1-\nu^{(h)}_1&(1<h\le r)
	\end{cases}\\
c_\pm(h)&=\sum_{j=1}^2(\nu^{(h)}_j-\nu^{(h-1)}_j)-(c_.(h-2)+c_+(h-2))\quad(2\le h\le r)
\end{align*}
where we have set $c_+(0)=0$ and used~\eqref{equation.chi} for the definition of $\chi$. 
Notice that $c_+(r)$ is not defined in the above formula.
It is determined by the fact that the total number of columns is $s$.

\subsection{Affine crystal structure}

In the last subsection we defined rigged configurations corresponding to $\pm$-diagrams. 
In Section~\ref{section.crystal background} we defined an involution $\mathfrak{S}$ on $\pm$-diagrams and saw
in Definition~\ref{definition.sigma} that it can be extended to the involution $\sigma$ on any element in $B^{r,s}$. 
In this vein, we make the following definition.

\begin{definition} \label{definition.sigma rc}
Let $(\nu,J)\in \RC(B^{r,s})$ with $B^{r,s}$ a KR crystal of type $D_n^{(1)}$ with $1\le r\le n-2$.
Choose a sequence $\mathbf{b}=(b_1,b_2,\ldots, b_k)$ with $b_i\in\{2,3,\ldots,n\}$ such that 
$e_{\mathbf{b}}(\nu,J) := e_{b_1} \cdots e_{b_k}(\nu,J)$ is $\{2,3,\ldots,n\}$-highest weight. Then define
\begin{equation}
	\sigmarc(\nu,J) := f_{\mathbf{b}'} \circ \gammarc \circ \; \mathfrak{S} \circ \gammarc^{-1} 
	\circ \;e_{\mathbf{b}}(\nu,J)\; ,
\end{equation}
where $\mathbf{b}'$ is the reverse of $\mathbf{b}$.
\end{definition}

Next we define $\sigmarc$ for $r=n-1,n$.

\begin{proposition}
The $\{2,\ldots,n\}$-highest weight elements of $\RC(B^{n-1,s})$ are given by 
\begin{equation} \label{J-ht for spin}
\begin{split}
	\nu&=((\hspace{2.5mm}j),(j),\ldots,(j),\emptyset)\\
	J&=((-j),(0),\ldots,(0),\emptyset)
\end{split}
\end{equation}
for $j=0,1,\ldots,s$, and those of $\RC(B^{n,s})$ by interchanging $(\nu^{(n-1)},J^{(n-1)})$ and
$(\nu^{(n)},J^{(n)})$ in each of the above elements. Here in the $j=0$ case $\nu$ and $J$ should be 
understood as a sequence of empty partitions.
\end{proposition}

\begin{proof}
It is not hard to check that the listed rigged configurations are indeed in $\RC(B^{n-1,s})$ and are
$\{2,\ldots,n\}$-highest weight. By~\cite[Sections 3.2 and 6.2]{FOS:2009} there are precisely $s+1$ such
highest weight vectors, completing the proof.
\end{proof}

\begin{definition} \label{definition.sigma spin}
The involution $\sigmarc : \RC(B^{n-1,s})\leftrightarrow\RC(B^{n,s})$ is defined by requiring 
\begin{enumerate}
\item $\sigmarc$ commutes with $e_i,f_i$ for $i=2,\ldots,n$, and
\item $\sigmarc$ interchanges the $\{2,\ldots,n\}$-highest weight element~\eqref{J-ht for spin} with 
	the one with $j$ replaced with $s-j$ and with $(\nu^{(n-1)},J^{(n-1)})$ and $
	(\nu^{(n)},J^{(n)})$ switched.
\end{enumerate}
\end{definition}

Our main theorem is the following affine crystal isomorphism.

\begin{theorem}\label{theorem.main}
Let $B^{r,s}$ be a KR crystal of type $D_n^{(1)}$ with $1\le r\le n$ and $s\ge 1$.
Then there is an affine crystal isomorphism
\[
	\iota : B^{r,s} \cong \RC(B^{r,s})
\]
extending the classical crystal isomorphism $\iota_0$ of Theorem~\ref{theorem.crystal isomorphism}
using the affine crystal operators
\begin{equation*}
	f_0 = \sigmarc \circ f_1 \circ \sigmarc \quad \text{and} \quad e_0 = \sigmarc \circ e_1 \circ \sigmarc.
\end{equation*}
\end{theorem}

\begin{proof}
For $1\le r\le n-2$, the result follows from Theorem~\ref{theorem.crystal isomorphism} and the fact that by 
Proposition~\ref{proposition.rc pm} $\sigma$ on $B^{r,s}$ and $\sigmarc$ on $\RC(B^{r,s})$ intertwine under
the classical crystal isomorphism $\iota_0$ by Definitions~\ref{definition.sigma} and~\ref{definition.sigma rc}. 
For $r=n-1,n$ the result follows by comparing Definition~\ref{definition.sigma spin} and~\cite[Definition 6.3]{FOS:2009} 
and using~\cite[Theorem 6.4]{FOS:2009}.
\end{proof}

The affine crystal isomorphism between Kirillov--Reshetikhin crystals and rigged configurations is also well-behaved
with respect to the grading by coenergy and cocharge.

\begin{theorem}\label{theorem:cc=D_for_single_rectangle}
Let $b\in B^{r,s}$ with $1\le r\le n$ and $s\ge 1$. Then
\begin{equation}
	\Db(b) = \cc(\iota(b)).
\end{equation}
\end{theorem}

\begin{proof}
By definition the coenergy is constant on classical components. By~\cite[Theorem 3.9]{S:2006} this is also true for cocharge.
Hence it suffices to prove the statement for highest weight elements. For this we first rewrite $\cc(\nu)$ 
in~\eqref{equation.cocharge} in terms of the vacancy numbers~\eqref{equation.vacancy}
\begin{equation} \label{equation.rewrite c}
	\cc(\nu)=\frac12\bigl(\sum_{(a,i) \in \HH} p^{(a)}_im^{(a)}_i+
	\sum_{a\in I_0,j,k\in\Z_{>0}}\min(j,k)L^{(a)}_jm^{(a)}_k\bigr).
\end{equation}
Note that $L^{(a)}_j=\chi(a=r)\chi(j=s)$ in our case.

For $r=n-1,n$ the statement holds since the only highest weight rigged configuration is the empty rigged configuration, which has
cocharge zero. Since there is only one classical component, the coenergy is also zero.

Now let $1\le r\le n-2$. 
Let $(\nu,J)$ be the rigged configuration corresponding to the highest weight $\lambda$ in Proposition~\ref{proposition.kleber}.
All riggings are zero, so that the contribution from the last term in~\eqref{equation.cocharge J} involving $J^{(a,i)}$ is zero.
Since all vacancy numbers are calculated to be zero by Proposition~\ref{proposition.kleber}, the contribution
from the first term in \eqref{equation.rewrite c} is zero. Hence we have
\[
	\cc(\nu,J)= \cc(\nu) = \frac12\sum_k\min(s,k)m^{(r)}_k=\frac{1}{2} |\nu^{(r)}| = \frac{1}{2} |\overline{\lambda}|
\]
again by Proposition~\ref{proposition.kleber}, which is equal to the number of vertical dominoes in 
$(s^r) \setminus \lambda$ and therefore agrees with the coenergy.
\end{proof}

\section{Combinatorial bijection} 
\label{section.cb}
In the previous sections we presented a bijection between a single Kirillov--Reshetikhin crystal
$B^{r,s}$ and the corresponding rigged configurations $\RC(B^{r,s})$, which representation theoretically
can be interpreted as an affine crystal isomorphism. In this section we give a combinatorial description
of this bijection. In fact, the definition of the combinatorial map can be given for arbitrary tensor products
not just a single KR crystal. It turns out that the description of the bijection involves a new kind of tableaux
of rectangular shape $(s^r)$ for the elements in $B^{r,s}$, which we call Kirillov--Reshetikhin tableaux,
instead of the usual Kashiwara--Nakashima tableaux
which in general are not rectangular. The procedure to go from Kashiwara--Nakashima tableaux to
the new rectangular tableaux is called the filling map and is the subject of Section~\ref{section.filling map}.
In Section~\ref{section.delta} we define the necessary combinatorial algorithms and conjecture that they
define a bijection $\Phi$ between a tensor product of crystals and rigged configurations. 
In Section~\ref{section.combinatorial bijection} it is proved that the combinatorial bijection $\Phi$
agrees with the crystal isomorphism $\iota$ on $\J$-highest weight elements for a single KR crystal.
We conclude in Section~\ref{section.conjectures} with conjectures and open questions.

\subsection{The filling map} \label{section.filling map}

In this section we define a filling map from KN tableaux of shape $\lambda$ (such that $B(\lambda)$ appears in $B^{r,s}$ 
of type $D_n^{(1)}$ with $1\le r\le n-2$ as a classical subcrystal) to tableaux of shape $(s^r)$ that appear in the 
combinatorially defined bijection between crystal elements and rigged configurations that will be described in 
Section~\ref{section.combinatorial bijection}.

Let the weight $\lambda$ be $k_r\varpi_r+k_{r-2}\varpi_{r-2}+\cdots$.
Let $k_c$ be the first odd integer (if it exists) in the sequence $k_{r-2},k_{r-4},\ldots$. If $k_c$ does not exist, set $c=-1$.
Then for the highest weight KN tableau $u_\lambda \in B(\lambda) \subset B^{r,s}$ we follow the procedure below, called
the \text{filling map}, to obtain a tableau $t$ of shape $(s^r)$.
The process proceeds by induction on the columns of $\lambda$ from left to right (according to $k_r,k_{r-2},\ldots$).
Recall that $u_\lambda$ is the tableau with 1s in row 1, 2s in row 2 etc..

\paragraph{\textbf{Step 0}}
The first $k_r$ columns of $t$ are the same as the first $k_r$ columns of $u_\lambda$ of height $r$, namely the columns
with entries $r  \cdots 2 1$.

\paragraph{\textbf{Step 1}}
For $k_h$ $(r>h\geq c)$, add the transpose of the following
rows to $t$ for $\lfloor k_h/2\rfloor$ times
$$\begin{array}{|c|c|c|c|c|c|c|c|}
\hline
1& 2& \cdots& h& \overline{r\rule{0pt}{6.1pt}}\rule{0pt}{10pt}
& \overline{r-1}& \cdots& \overline{h+1}\\
\hline
1& 2& \cdots& h& h+1& h+2& \cdots& r\\
\hline
\end{array}$$

\paragraph{\textbf{Step 2}}
For each column of $\lambda$ of height $h$ with $c>h$, add the transpose of the following row to $t$
$$\begin{array}{|c|c|c|c|c|c|c|c|c|c|c|c|}
\hline
1& \cdots& h-1& h& r-(x-h-2)&\cdots& r-1& r&
\overline{r\rule{0pt}{6.2pt}}\rule{0pt}{10pt}
& \cdots& \overline{x+1}& \overline{x\rule{0pt}{6.2pt}}\\
\hline
\end{array}\,.$$
Here $x$ is defined as follows.
As the initial condition, set $x=c+1$.
After putting the first column with this $x$,
we recursively redefine $x$ as follows.
Assume that the previous column was of height $h'$.
Then for the next filling, set $x=(h'+1)$-th letter of the
previous column (i.e., the top of the filled letters).

\paragraph{\textbf{Step 3}}
If $c>-1$, let $x$ be the final one obtained in Step 2.
Then the rightmost column is the transpose of
$$\begin{array}{|c|c|c|c|c|c|c|c|}
\hline
1& 2& \cdots& (r+x-1)/2& \overline{(r+x-1)/2}\rule{0pt}{11pt}&
\cdots& \overline{x+1}& \overline{x\rule{0pt}{6.5pt}}\\
\hline
\end{array}\,.$$

The final result $t$ is the filling of $u_\lambda$, denoted by $\fillmap(u_\lambda)$.

\begin{definition}\label{def:KRtableaux}
Let $b$ be the KN tableau representation of an element of $B(\lambda)$.
Let $u_\lambda=e_{a_k}\cdots e_{a_2}e_{a_1}(b)$ be the corresponding
$I_0$-highest weight vector (in particular $a_i \in I_0$ for $1\le i\le k$).
Then the {\bf Kirillov--Reshetikhin (KR) tableau} representation of $b$ is defined by
$\fillmap (b)=f_{a_1}f_{a_2}\cdots f_{a_k}\fillmap(u_\lambda)$.
Here the action of $e_i,f_i$ ($i\in I_0$) on KR tableaux is defined in a similar way to the action
on KN tableaux by reading the tableau columnwise and using the signature rule on the corresponding word.
\end{definition}

\begin{remark}
We can check that both $u_\lambda$ and $\fillmap (u_\lambda)$
are $I_0$-highest weight elements of weight $\lambda$.
Therefore the above $\fillmap (b)$ is always well-defined.
\end{remark}

\begin{example}
We list the filling map for the $I_0$-highest weight elements of weights
$(k_{10},k_8,k_6,k_4,k_2,k_0)=(2,\underline{3},0,0,3,1)$,
$(2,\underline{1},0,0,3,1)$, $(2,0,\underline{1},0,3,1)$
from left to right, respectively. The underlined letters correspond to $k_c$.
Here we set $r=12$ and $k_{12}=0$.
We color the cells $1,2,\ldots,i$ of type
$k_i$ corresponding to Step 1, 2, and 3
by pink, yellow and green, respectively.
\begin{center}
  \includegraphics[width=44mm]{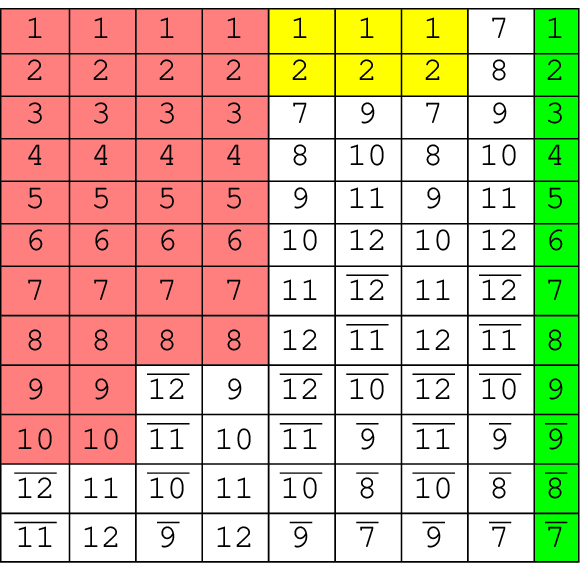}\hspace{2mm}
  \includegraphics[width=34mm]{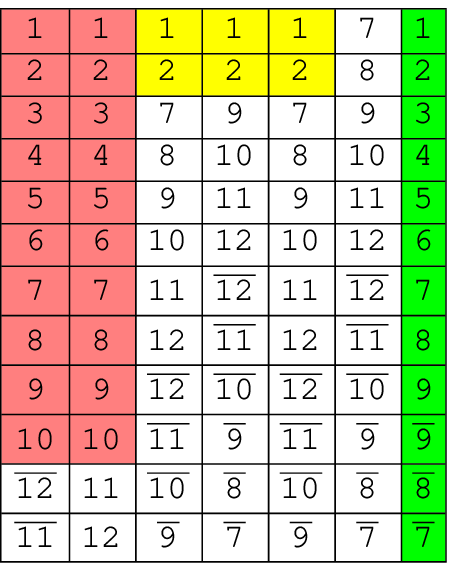}\hspace{2mm}
  \includegraphics[width=34mm]{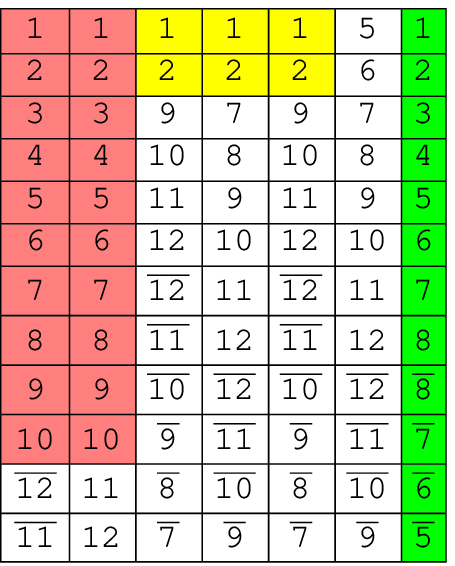}
\end{center}
\end{example}

\begin{remark} \label{remark.induction}
Note that if we start with a highest weight element $u_\lambda \in B(\lambda) \subset B^{r,s}$ for a KR crystal
of type $D_n^{(1)}$ with $1\le r\le n-2$, then $t'$ obtained from $t=\fillmap(u_\lambda)$ by removing the leftmost
column is in the image of the filling map for a different weight $\lambda'$, namely $t'=\fillmap(u_{\lambda'})$, where 
$u_{\lambda'}$ is the highest weight element in $B(\lambda')\subset B^{r,s-1}$. More precisely:
\begin{enumerate}
\item If $\lambda$ has a column of height $r$, then $\lambda'$ is obtained from $\lambda$ by removing a column of height $r$. 
\item If the two leftmost columns of $\lambda$ are of height $h$ and $h'$ with $r>h\ge h'$, then $\lambda'$ is
obtained from $\lambda$ by removing its leftmost column of height $h$ and replacing the next column of height $h'$ by 
a column of height $h'+(r-h)$.
\item If $\lambda$ has only one column, $\lambda'$ is the empty partition (or zero weight).
\end{enumerate}
Step 0 in the algorithm for the filling map corresponds to (i), Step 1 corresponds to a combination of (ii) with $h=h'$ followed by (i),
Step 2 to (ii), and Step 3 to (iii).
\end{remark}

\subsection{Operations on rigged configurations} \label{section.delta}
First we define the basic operation which we call $\delta$
$$\delta:(\nu,J)\longmapsto \{(\nu',J'),k\} \; ,$$
where $(\nu,J)$ and $(\nu',J')$ are rigged configurations
and $k\in\{1,2,\ldots,n,\bar{n},\ldots,\bar{2},\bar{1}\}$.
This map is a slight generalization of the type $D$ algorithm of~\cite{OSS:2003a} from single boxes to rectangles.
The $\delta$ operation constitutes an elementary step of our main map $\Phi$.

\begin{definition}
Suppose that $(\nu,J) \in \RC(B)$, where $B$ is a tensor product of KR crystals and the leftmost factor of $B$ is
$B^{a,l}$ where $1\leq a\leq n-2$. Then the map $\delta^{(a)}_l$
$$\delta^{(a)}_l:(\nu,J)\longmapsto \{(\nu',J'),k\}$$
is defined by the following procedure.
Set $\ell^{(a-1)}=l$.
\begin{enumerate}
\item[(1)]
For $a<i\leq n-2$, assume that $\ell^{(i-1)}$ is already determined.
Then we search for the shortest singular string in $(\nu,J)^{(i)}$
that is longer than or equal to $\ell^{(i-1)}$.
\begin{enumerate}
\item If there exists such a string, set $\ell^{(i)}$ to be
the length of the selected string and continue the process recursively.
If there is more than one such string, choose any of them.
\item If there is no such string, set $\ell^{(i)}=\infty$,
$k=i$ and stop.
\end{enumerate}
\item[(2)]
Suppose that $\ell^{(n-2)}<\infty$.
Then we search for the shortest singular string in $(\nu,J)^{(n-1)}$
(resp. $(\nu,J)^{(n)}$) that is longer than or equal to $\ell^{(n-2)}$
and define $\ell^{(n-1)}$ (resp. $\ell^{(n)}$) similarly.
\begin{enumerate}
\item If $\ell^{(n-1)}=\infty$ and $\ell^{(n)}=\infty$,
set $k=n-1$ and stop.
\item If $\ell^{(n-1)}<\infty$ and $\ell^{(n)}=\infty$,
set $k=n$ and stop.
\item If $\ell^{(n-1)}=\infty$ and $\ell^{(n)}<\infty$,
set $k=\bar{n}$ and stop.
\item If $\ell^{(n-1)}<\infty$ and $\ell^{(n)}<\infty$,
set $\bar{\ell}^{(n-1)}=\max(\ell^{(n-1)},\ell^{(n)})$ and continue.
\end{enumerate}
\item[(3)]
For $1\leq i\leq n-2$, assume that $\bar{\ell}^{(i+1)}$ is already defined.
Then we search for the shortest singular string in $(\nu,J)^{(i)}$
that is longer than or equal to $\bar{\ell}^{(i+1)}$ and has not yet been selected
as $\ell^{(i)}$.
Define $\bar{\ell}^{(i)}$ similarly.
If $\bar{\ell}^{(i)}=\infty$, set $k=\overline{i+1}$ and stop.
Otherwise continue.
If $\bar{\ell}^{(1)}<\infty$, set $k=\bar{1}$ and stop.
\item[(4)]
Once the process has stopped, remove the rightmost box of each selected row
specified by $\ell^{(i)}$ or $\bar{\ell}^{(i)}$.
The result gives the output $\nu'$.
\item[(5)]
Define the new riggings $J'$ as follows.
For the rows that are not selected by $\ell^{(i)}$ or $\bar{\ell}^{(i)}$,
take the corresponding riggings from $J$.
For the remaining parts, replace one $B^{a,l}$ in $B$ by
$B^{a-1,1}\otimes B^{a,l-1}$ (in anti-Kashiwara convention for tensor products). 
Denote the result by $B'$. Use $B'$ to compute all the vacancy numbers for $\nu'$.
Then the remaining riggings are defined so that all the corresponding
rows become singular with respect to the new vacancy number.
\end{enumerate}
\end{definition}

We remark that the resulting rigged configuration $(\nu',J')$
is associated with the tensor product $B'$.
For the sake of simplicity, we sometimes omit subscript or superscript of
$\delta^{(a)}_l$.
We write $\delta_2\delta_1(\nu,J)$ etc. for repeated applications of $\delta$
on the rigged configurations.

\begin{definition}
For a given rigged configuration $(\nu,J) \in \RC(B)$, where 
$B=B^{r_1,s_1}\otimes B^{r_2,s_2}\otimes \cdots\otimes B^{r_L,s_L}$ is a tensor product of KR crystals,
define the map $\Phi_B$ (sometimes also just denoted $\Phi$)
\begin{equation*}
\begin{split}
	\Phi_B \;:\;  \RC(B) &\longrightarrow B\\
	(\nu,J) &\longmapsto b
\end{split}
\end{equation*}
as follows.
Here $b$ is a filling of the rectangular shapes $(s_1^{r_1})$,
$(s_2^{r_2})$, $\ldots$, $(s_L^{r_L})$ (from left to right)
by the letters $k\in\{1,2,\ldots,n,\bar{n},\ldots,\bar{2},\bar{1}\}$.
\begin{enumerate}
\item[(1)] Suppose 
$\delta^{(1)}_1\cdots\delta^{(r_1-1)}_1\delta^{(r_1)}_{s_1}(\nu,J)=(\nu',J')$
yields the sequence of letters $k^{(r_1)},k^{(r_1-1)},\cdots,k^{(1)}$
($k^{(a)}$ corresponds to $\delta^{(a)}$).
Put the transpose of the row
$\begin{array}{|c|c|c|c|}
\hline
k^{(1)}& k^{(2)}& \cdots& k^{(r_1)}\rule{0pt}{12pt}\\
\hline
\end{array}$
as the leftmost column of the rectangle $(s_1^{r_1})$.
\item[(2)] Continue the previous step for
$\delta^{(1)}_1\cdots\delta^{(r_1-1)}_1\delta^{(r_1)}_{s_1-1}(\nu',J')=(\nu'',J'')$
and fill the second column for the rectangle $(s_1^{r_1})$ with the produced letters.
Repeat the process until all places of $(s_1^{r_1})$ are filled.
\item[(3)] Repeat the previous two steps for the remaining rectangles
$(s_2^{r_2})$, $(s_3^{r_3})$, $\ldots$, $(s_L^{r_L})$.
\end{enumerate}
\end{definition}

Now we propose the basic conjecture about the above map $\Phi$.

\begin{conjecture}\label{conj:main}
The map $\Phi$ gives a bijection between the set of rigged configurations
$\RC(B)$ and the tensor product of KR crystals $B$.
Here we identify the rectangular tableaux obtained by $\Phi$ as
the KR tableaux representation of the elements of crystals.
\end{conjecture}

In Theorem~\ref{theorem:single_highest_case} below we will prove this conjecture
for the highest weight elements of $B^{r,s}$.
We stress that there is no a priori reason that the combinatorially defined bijection $\Phi$
admits crystal structure that appeared in the definition of the KR tableaux.

\begin{example}
Let us consider the following rigged configuration of type
$B^{2,2}\otimes B^{3,1}\otimes B^{2,1}\otimes
B^{1,3}\otimes B^{1,1}\otimes B^{1,1}\otimes B^{1,1}$
(in anti-Kashiwara convention for the tensor products) of type $D^{(1)}_5$.
\begin{center}
\unitlength 10pt
\begin{picture}(35,4.5)
\put(0.1,0.1){0}
\put(0.1,1.1){0}
\put(0.1,2.1){0}
\put(0.1,3.1){1}
\put(1,0){$\Yboxdim10pt\yng(4,1,1,1)$}
\put(2.3,0.1){0}
\put(2.3,1.1){0}
\put(2.3,2.1){0}
\put(5.3,3.1){1}
\put(7.1,0.1){2}
\put(7.1,1.1){0}
\put(7.1,2.1){0}
\put(7.1,3.1){0}
\put(9.1,2.15){$\times$}
\put(8.0,0){$\Yboxdim10pt\yng(5,2,2,1)$}
\put(9.3,0.1){2}
\put(10.1,1.1){$-2$}
\put(10.3,2.1){0}
\put(13.1,3.1){$-2$}
\put(15.1,0.1){1}
\put(15.1,1.1){1}
\put(15.1,2.1){1}
\put(15.1,3.1){0}
\put(18.05,2.15){$\times$}
\put(20.0,3.1){$\times$}
\put(16.0,0){$\Yboxdim10pt\yng(5,3,1,1)$}
\put(17.3,0.1){1}
\put(17.3,1.1){1}
\put(19.2,2.1){1}
\put(21.1,3.1){0}
\put(23.1,2.1){0}
\put(22.4,3.1){$-1$}
\put(27,3.15){$\times$}
\put(24,2){$\Yboxdim10pt\yng(4,1)$}
\put(25.3,2.1){0}
\put(28.1,3.1){$-1$}
\put(30.1,2.1){0}
\put(30.1,3.1){0}
\put(33.05,3.15){$\times$}
\put(31,2){$\Yboxdim10pt\yng(3,1)$}
\put(32.3,2.1){0}
\put(34.3,3.1){0}
\end{picture}
\end{center}
We consider $B^{2,2}$ first.
Then we remove the boxes indicated by ``$\times$" in the above diagram
($\ell^{(2)}=2$, $\ell^{(3)}=3$, $\ell^{(4)}=4$, $\ell^{(5)}=3$,
$\bar{\ell}^{(4)}=4$, $\bar{\ell}^{(3)}=5$)
and obtain the letter $\bar{3}$ as the output.
For the next step we change $B^{2,2}$ into $B^{1,1}\otimes B^{2,1}$
and compute the vacancy number again.
We can continue this process as follows:
\begin{center}
\unitlength 10pt
\begin{picture}(35,4.5)
\put(0.1,0.1){1}
\put(0.1,1.1){1}
\put(0.1,2.1){1}
\put(0.1,3.1){1}
\put(4.05,3.1){$\times$}
\put(1,0){$\Yboxdim10pt\yng(4,1,1,1)$}
\put(2.3,0.1){0}
\put(2.3,1.1){0}
\put(2.3,2.1){0}
\put(5.3,3.1){1}
\put(7.1,0.1){2}
\put(7.1,1.1){2}
\put(7.1,2.1){1}
\put(6.4,3.1){$-1$}
\put(8.0,0){$\Yboxdim10pt\yng(5,2,1,1)$}
\put(9.3,0.1){2}
\put(9.3,1.1){2}
\put(10.1,2.1){$-2$}
\put(13.1,3.1){$-2$}
\put(15.1,0.1){1}
\put(15.1,1.1){1}
\put(15.1,2.1){1}
\put(15.1,3.1){0}
\put(16.0,0){$\Yboxdim10pt\yng(4,2,1,1)$}
\put(17.3,0.1){1}
\put(17.3,1.1){1}
\put(18.2,2.1){1}
\put(20.1,3.1){0}
\put(23.1,2.1){0}
\put(22.4,3.1){$-1$}
\put(24,2){$\Yboxdim10pt\yng(3,1)$}
\put(25.3,2.1){0}
\put(27.1,3.1){$-1$}
\put(30.1,2.1){0}
\put(30.1,3.1){0}
\put(31,2){$\Yboxdim10pt\yng(2,1)$}
\put(32.3,2.1){0}
\put(33.3,3.1){0}
\end{picture}
\end{center}
\begin{center}
\unitlength 10pt
\begin{picture}(35,4.5)
\put(0.1,0.1){0}
\put(0.1,1.1){0}
\put(0.1,2.1){0}
\put(0.1,3.1){1}
\put(1.1,0.2){$\times$}
\put(1,0){$\Yboxdim10pt\yng(3,1,1,1)$}
\put(2.3,0.1){0}
\put(2.3,1.1){0}
\put(2.3,2.1){0}
\put(4.3,3.1){1}
\put(7.1,0.1){2}
\put(7.1,1.1){2}
\put(7.1,2.1){1}
\put(6.3,3.1){$-2$}
\put(8.1,0.2){$\times$}
\put(8.1,1.2){$\times$}
\put(8.0,0){$\Yboxdim10pt\yng(5,2,1,1)$}
\put(9.3,0.1){2}
\put(9.3,1.1){2}
\put(10.1,2.1){$-2$}
\put(13.1,3.1){$-2$}
\put(15.1,0.1){1}
\put(15.1,1.1){1}
\put(15.1,2.1){1}
\put(15.1,3.1){0}
\put(16.1,0.2){$\times$}
\put(16.1,1.2){$\times$}
\put(16.0,0){$\Yboxdim10pt\yng(4,2,1,1)$}
\put(17.3,0.1){1}
\put(17.3,1.1){1}
\put(18.1,2.1){1}
\put(20.1,3.1){0}
\put(23.1,2.1){0}
\put(22.4,3.1){$-1$}
\put(24.1,2.2){$\times$}
\put(24,2){$\Yboxdim10pt\yng(3,1)$}
\put(25.3,2.1){0}
\put(27.1,3.1){$-1$}
\put(30.1,2.1){0}
\put(30.1,3.1){0}
\put(31.1,2.2){$\times$}
\put(31,2){$\Yboxdim10pt\yng(2,1)$}
\put(32.3,2.1){0}
\put(33.3,3.1){0}
\end{picture}
\end{center}
\begin{center}
\unitlength 10pt
\begin{picture}(35,3.5)
\put(0.1,0.1){1}
\put(0.1,1.1){1}
\put(0.1,2.1){2}
\put(1,0){$\Yboxdim10pt\yng(3,1,1)$}
\put(2.3,0.1){0}
\put(2.3,1.1){0}
\put(4.3,2.1){1}
\put(7.1,1.1){1}
\put(6.3,2.1){$-2$}
\put(8.0,1){$\Yboxdim10pt\yng(5,2)$}
\put(10.1,1.1){$-2$}
\put(13.1,2.1){$-2$}
\put(15.1,1.1){1}
\put(15.1,2.1){0}
\put(16.0,1){$\Yboxdim10pt\yng(4,2)$}
\put(18.1,1.1){1}
\put(20.1,2.1){0}
\put(22.4,2.1){$-1$}
\put(24,2){$\Yboxdim10pt\yng(3)$}
\put(27.1,2.1){$-1$}
\put(30.1,2.1){0}
\put(31,2){$\Yboxdim10pt\yng(2)$}
\put(33.3,2.1){0}
\end{picture}
\end{center}
\begin{center}
\unitlength 10pt
\begin{picture}(35,3.5)
\put(0.1,0.1){0}
\put(0.1,1.1){0}
\put(0.1,2.1){1}
\put(1,0){$\Yboxdim10pt\yng(3,1,1)$}
\put(2.3,0.1){0}
\put(2.3,1.1){0}
\put(4.3,2.1){1}
\put(7.1,1.1){1}
\put(6.3,2.1){$-2$}
\put(8.0,1){$\Yboxdim10pt\yng(5,2)$}
\put(10.1,1.1){$-2$}
\put(13.1,2.1){$-2$}
\put(15.1,1.1){1}
\put(15.1,2.1){0}
\put(16.0,1){$\Yboxdim10pt\yng(4,2)$}
\put(18.1,1.1){1}
\put(20.1,2.1){0}
\put(22.4,2.1){$-1$}
\put(24,2){$\Yboxdim10pt\yng(3)$}
\put(27.1,2.1){$-1$}
\put(30.1,2.1){0}
\put(31,2){$\Yboxdim10pt\yng(2)$}
\put(33.3,2.1){0}
\end{picture}
\end{center}
The outputs are 2, $\bar{1}$ and 1 respectively, which give the tableau
$\Yvcentermath1\young(21,\mthree\mone)\,$.
For the reader's convenience, we remark that in the second rigged configuration, we have
$\ell^{(2)}=\ell^{(3)}=\ell^{(4)}=\ell^{(5)}=\bar{\ell}^{(4)}=\bar{\ell}^{(3)}
=\bar{\ell}^{(2)}=\bar{\ell}^{(1)}=1$ and obtain the letter $\bar{1}$ as the output.
If we continue the process entirely we obtain the following
tensor product
$$b=\Yvcentermath1\young(21,\mthree\mone)\otimes
\young(1,3,\mtwo)\otimes
\young(2,4)\otimes
\young(15\mthree)\otimes
\young(3)\otimes
\young(1)\otimes\young(3)\,.$$
Let us convert the result into the KN tableaux representation.
Take $\Yvcentermath1\young(21,\mthree\mone)\,$ as an example.
Then we have
$e_2e_3e_5e_4e_3e_1\Yvcentermath1\young(21,\mthree\mone)=\young(11,2\mone)
\xrightarrow{\fillmap^{-1}}\young(1,2)$ and since
$f_1f_3f_4f_5f_3f_2\left(\Yvcentermath1\young(1,2)\right)
=\Yvcentermath1\young(2,\mthree)$ we get the
identification $\fillmap\left(\Yvcentermath1\young(2,\mthree)\right)=
\Yvcentermath1\young(21,\mthree\mone)$ in $B^{2,2}$.
In particular, we have checked that the resulting tableau is indeed a KR tableau.
To summarize, in the KN tableaux representation, the result looks as follows
$$\fillmap^{-1}(b)=
\Yvcentermath1\young(2,\mthree)\otimes
\young(1,3,\mtwo)\otimes
\young(2,4)\otimes
\young(15\mthree)\otimes
\young(3)\otimes
\young(1)\otimes\young(3)\,.$$
\end{example}

\subsection{The combinatorial bijection} \label{section.combinatorial bijection}

In this section we show that the combinatorially defined map $\Phi$ between $\RC(B^{r,s})$
and $B^{r,s}$ agrees with our previous affine crystal isomorphism $\iota$ from Theorem~\ref{theorem.main}
on highest weight elements.
Thus we prove Conjecture~\ref{conj:main} for the highest weight elements of $B^{r,s}$.

\begin{theorem}\label{theorem:single_highest_case}
Let $B^{r,s}$ be a KR crystal of type $D_n^{(1)}$ and $1\le r\le n-2$, $s \ge 1$. We have
\[
	\Phi = \fillmap \circ \iota^{-1}
\]
on highest weight elements.
\end{theorem}

\begin{proof}
Let $(\nu,J)$ be a highest weight element of weight $\lambda$ as in Proposition~\ref{proposition.kleber}.
Since all highest weights appear with multiplicity one in $B^{r,s}$ we immediately know the corresponding
KN tableau $b = \iota^{-1}(\nu,J)$. Let $t = \fillmap(b)$.
We will prove the claim by induction in four steps, corresponding to the steps in the definition of the filling map
of Section~\ref{section.filling map}.
Let $h$ be the largest integer satisfying $k_h>0$ and
$h'$ be the second largest integer satisfying $k_{h'}>0$ (if it exists).
\begin{enumerate}
\item[\textbf{Step 0.}] Suppose $h=r$.
We show that the removal of one column on $(\nu,J)$ via the combinatorial algorithm
for $\Phi$ corresponds to the removal of the leftmost column of $t$ (which is $r \cdots 21$).
\item[\textbf{Step 1.}] Suppose $c\le h<r$ and $k_h\geq 2$.
We show that the removal of two columns on $(\nu,J)$ via the combinatorial
algorithm for $\Phi$ corresponds to the removal of the two leftmost columns of $t$.
\item[\textbf{Step 2.}] Suppose $h=c$ and $k_c=1$ and $k_{h'}>0$. We show that the removal of one column on $(\nu,J)$ 
by the combinatorial algorithm defining $\Phi$ corresponds to the removal of the leftmost column of $t$.
\item[\textbf{Step 3.}] Suppose $k_c =1$ and all other $k_i=0$. Then the combinatorial algorithm for $\Phi$ 
for a single column agrees with the filling map.
\end{enumerate}

To prove Step 0, recall that the algorithm for the first application of $\delta$ in the definition of $\Phi$ for $B^{r,s}$
demands that the singular strings are of length at least $s$. However by Proposition~\ref{proposition.kleber}, the partitions
$\nu^{(a)}$ in $(\nu,J)$ have no parts of length $s$.
In the next step, we use $B^{r-1,1} \otimes B^{r,s-1}$ (in anti-Kashiwara convention for the tensor products)
for the computations of the vacancy numbers.
Then the vacancy numbers for $(\delta (\nu,J))^{(r-1)}$ become all one, so that
there is no singular string and thus we cannot select a row.
The same situation also holds for all $\delta^i$ $(1\leq i\leq r)$.
Hence after $r$ applications of $\delta$ (which corresponds to the removal of one column)
the corresponding letters are $r \cdots 21$ which are precisely the first column of $t$.

For Step 1, we claim that the lengths of the selected strings by the first $1\le i\le r-h$ applications of $\delta$ are given as follows,
where $h$ is the height of the first column of $\lambda$ (or equivalently $b$):
\begin{equation} \label{equation.ell}
\begin{split}
\ell^{(a)} &= \begin{cases}
	s-1 & \text{for $a=r-i+1,\ldots,r-1$,}\\
	s    & \text{for $a=r,\ldots,n,\overline{n-2},\ldots,\overline{h+i}$,}
	\end{cases}
	\qquad \quad \text{for $i$ odd,}\\
\ell^{(a)} &= \begin{cases}
	s-1 & \text{for $a=r-i+1,\ldots,n,\overline{n-2},\ldots,\overline{r}$,}\\
	s    & \text{for $a=\overline{r-1},\ldots,\overline{h+i}$,}
	\end{cases}
	\quad \text{for $i$ even,}
\end{split}
\end{equation}
with all other $\ell^{(a)}=\infty$. This follows from the fact that for the first application of $\delta$ the selected strings need to
be of length at least $s$ and they exist by Proposition~\ref{proposition.kleber} and are singular. For $1< i\le r$, the strings of 
length strictly less than $s-1$ in $(\delta^{i-1}(\nu,J))^{(r-i+1)}$ are nonsingular (in fact the vacancy number is one and the riggings 
are zero), but there is a singular string of length $s-1$. Hence $\ell^{(r-i+1)}=s-1$ in all cases.
There exist singular strings of length $s-1$ of rigging and vacancy number 0 for all $(\delta^{i-1}(\nu,J))^{(a)}$
for $r-i<a\le r-1$ and they are chosen. For $i$ even there exist singular strings of length $s-1$ in $(\delta^{i-1}(\nu,J))^{(a)}$
for $a=r,\ldots,n, \overline{n-2},\ldots,\overline{r}$. However, for $i$ odd there exist no such parts in 
$(\delta^{i-1}(\nu,J))^{(r)}$, so that the singular strings of length $s$ are chosen. Altogether this proves~\eqref{equation.ell}.

It is not hard to check using~\eqref{equation.ell} that the partitions in $\delta^{r-h}(\nu,J)$ are of the form of 
Proposition~\ref{proposition.kleber} for weight given by $\lambda$ with the first two columns of height $h$ removed from $\lambda$.
Also, the letters produced by the first $r-h$ applications of $\delta$ are $\overline{r-h}, \overline{r-h+1}, \ldots, \overline{r}$.
There are no singular strings in $(\delta^{r-h}(\nu,J))^{(h)}$ so that the remaining $h$ applications of $\delta$
do not change $\delta^{r-h}(\nu,J)$ and the produced letters are $h, h-1, \ldots, 1$. Comparing with the filling map
of Section~\ref{section.filling map} this indeed produces the first column of $t$.

To remove the next column, note that the selected singular strings need to be of length at least $s-1$ in the $r$-th rigged 
partition. However, since the shape is now $\lambda$ with the first two columns removed, there are no such singular
string. Hence the next $r$ applications of $\delta$ do not change the rigged configuration $\delta^r(\nu,J)$ and
the produced letters are $r,r-1,\ldots,1$ which corresponds to the second column of $t$ from the left.

In Step 2, let $h$ and $h'$ denote the height of the two leftmost columns of $\lambda$, respectively.
For Step 2 we have $h>h'$ since the first column is of height $h=c$ and since $k_c=1$, the next column of height
$h'$ must be strictly smaller.
For the first $i$ applications of $\delta$ for $1\leq i\leq r-h$, by similar arguments as in Step 1 the same singular strings 
are chosen as in Equation~\eqref{equation.ell}. This produces the letters $\overline{r-h}, \overline{r-h+1}, \ldots,
\overline{r}$.

We claim that the length of the chosen strings by $\delta^i$ for $r-h<i\leq r-h'$ is given by
\begin{equation} \label{equation.ell2}
	\ell^{(a)} = s-1 \qquad \text{for $a=r-i+1,r-i+2,\ldots,2r-h-i$}
\end{equation}
and all other $\ell^{(a)}=\infty$.
To prove~\eqref{equation.ell2}, note that the partitions in $\delta^{r-h}(\nu,J)$
are obtained from $\nu^{(a)}$ by removing the $s$-th column (if it exists) and changing the $(s-1)$-st column
as follows:
\begin{itemize}
\item For $a=n-1,n$, the height of the $(s-1)$-st column in $\nu^{(a)}$ becomes $(h-h')/2$.
\item For $h\leq a\leq n-2$, the height of the $(s-1)$-st column in $\nu^{(a)}$ becomes $h-h'$. 
\item For $h'<a<h$, the height of the $(s-1)$-st column in $\nu^{(a)}$ becomes $a-h'$.
\item For $1\le a\le h'$, there is no $(s-1)$-st column.
\end{itemize}
The riggings are all 0.
The vacancy numbers are $p^{(r)}_{s-1}=1$,
$p^{(h)}_l=1$ where $l\leq s-2$ and 0 otherwise.
The next step $i=r-h+1$ we begin to choose rows of $(\delta^{r-h}(\nu,J))^{(h)}$.
The only singular strings in $(\delta^{r-h}(\nu,J))^{(h)}$ are of length $s-1$,
which implies $\ell^{(a)}=s-1$ for $h\leq a\leq r-1$.
Since there are no singular strings in $(\delta^{r-h}(\nu,J))^{(r)}$ which are longer than or equal to $s-1$, we 
have $\ell^{(r)}=\infty$. This agrees with~\eqref{equation.ell2} for $i=r-h+1$.
Therefore after finishing the step $i=r-h+1$,
the strings of length $s-1$ in $(\delta^{r-h+1}(\nu,J))^{(r-1)}$ and
the strings of length strictly less than $s-1$ in $(\delta^{r-h+1}(\nu,J))^{(h-1)}$ become non-singular.
Thus in the next step $i=r-h+2$, we have to choose a singular string of length
$s-1$ in $(\delta^{r-h+1}(\nu,J))^{(h-1)}$ and stop at $(\delta^{r-h+1}(\nu,J))^{(r-2)}$.
We can continue this process inductively until the $i=(r-h')$-th step.
The letters produced in steps $r-h<i\le r-h'$ are $r,r-1,\ldots,h'+(r-h)+1$.

After the application of $\delta^{r-h'}$, there are no singular strings in $(\delta^{r-h'}(\nu,J))^{(h')}$.
Thus there are no more strings removed by $\delta^i$ for $r-h'<i\le r$ and the produced letters are $h',h'-1,\ldots,1$.
Comparing with Step 2 of the algorithm in Section~\ref{section.filling map} the produced letters precisely form the leftmost column
of $t$, which is hence removed. Note also that the resulting rigged configuration $\delta^r(\nu,J)$ is of the form
of Proposition~\ref{proposition.kleber} for a new weight $\lambda'$ as given in Remark~\ref{remark.induction}.

Step 3 follows from~\cite{S:2005} since in this case the filling map of Section~\ref{section.filling map}
agrees with the one used in~\cite{S:2005}.
\end{proof}

\subsection{Conjectures and open questions} \label{section.conjectures}
Let us raise several conjectures about our map $\Phi$.

\begin{conjecture}\label{conjecture.broccoli}
$\Phi$ commutes with the Kashiwara operators $f_i,e_i$ for $i\neq 0$.
\end{conjecture}

The analogue of this conjecture for type $A$ was proved in~\cite{DS:2006}.
We expect that similar methods might work for the proof in type $D$.

\begin{conjecture}\label{conjecture.R}
Let $(\nu,J) \in \RC(B)$, where $B = B_1 \otimes B_2$ is a two-fold tensor product of 
Kirillov--Reshetikhin crystals. Taking the other order $B' = B_2 \otimes B_1$, we have
\[
	R(\Phi_B(\nu,J)) = \Phi_{B'}(\nu,J).
\] 
\end{conjecture}

\begin{remark}
Conjecture~\ref{conjecture.R} implies a similar statement for an arbitrary number of tensor factors by
applying a sequence of combinatorial $R$-matrices in an appropriate way.
\end{remark}

Conjecture~\ref{conjecture.R} gives a generalization of
the inverse scattering transform of the box-ball systems~\cite{KOSTY:2006,KSY:2011}
by the same argument.
In this setting, each row of a rigged configuration is regarded as a soliton
of the size equal to the length of the corresponding row.
For $\bigotimes_iB^{1,s_i}$ of type $A$, this point of view is also confirmed~\cite{KSY:2006,Sak:2008} 
by explicitly making a connection with the soliton solution for the KP equation~\cite{JM:1983}.

Let us provide some evidence for Conjecture~\ref{conjecture.R} by examples.

\begin{example}
Let us consider type $D_5^{(1)}$ and $R: B^{2,3}\otimes B^{3,2}\simeq B^{3,2}\otimes B^{2,3}$, where we use the 
anti-Kashiwara convention for tensor products, with weight $\varpi_3+\varpi_2+\varpi_1$.
Take the rigged configuration
\begin{center}
\unitlength 10pt
\begin{picture}(25,4.5)
\put(0.2,3.1){0}
\put(1,3){$\Yboxdim10pt\yng(2)$}
\put(3.2,3.1){0}
\put(5.2,1.1){0}
\put(5.2,2.1){0}
\put(5.2,3.1){1}
\put(6,1){$\Yboxdim10pt\yng(3,1,1)$}
\put(7.2,1.1){0}
\put(7.2,2.1){0}
\put(9.2,3.1){0}
\put(11.2,0.1){0}
\put(11.2,1.1){0}
\put(11.2,2.1){0}
\put(11.2,3.1){0}
\put(12,0){$\Yboxdim10pt\yng(2,2,1,1)$}
\put(13.2,0.1){0}
\put(13.2,1.1){0}
\put(14.2,2.1){0}
\put(14.2,3.1){0}
\put(16.2,2.1){0}
\put(16.2,3.1){0}
\put(17,2){$\Yboxdim10pt\yng(2,1)$}
\put(18.2,2.1){0}
\put(19.2,3.1){0}
\put(21.2,2.1){0}
\put(21.2,3.1){0}
\put(22,2){$\Yboxdim10pt\yng(2,1)$}
\put(23.2,2.1){0}
\put(24.2,3.1){0}
\end{picture}
\end{center}
Under the combinatorial bijection $\Phi$ between rigged configurations and crystal paths, $(\nu,J)$ corresponds
to the two tensor products (using the two orderings, respectively):
$$\Yvcentermath1\young(11\mthree,23\mone)\otimes\young(11,22,3\mtwo)\,,\qquad \young(11,33,\mthree\mone)\otimes\young(111,22\mone)\,.$$
We now convert these KR tableaux into usual KN tableaux to be able to compare the result
with the algorithm for the combinatorial $R$-matrix presented in~\cite{LOS:2011a}.
Take for example $\Yvcentermath1\young(11,33,\mthree\mone)$, which lies in the crystal component of the highest
weight vector $\Yvcentermath1\young(11,22,3\mtwo)$. Under the filling map
\[
	\Yvcentermath1 \fillmap\left(\young(11,2,3) \right) = \young(11,22,3\mtwo) \; .
\]	
Applying the same lowering Kashiwara operators hence yields the identification
\[
	\Yvcentermath1 \fillmap \left(\young(13,3,\mthree)\right) = \young(11,33,\mthree\mone) \; .
\]
By similar computations, we obtain the two tensor products in the
KN tableaux representations:
\[
	\Yvcentermath1\young(11\mthree,23\mone)\otimes\young(11,2,3)\simeq \young(13,3,\mthree)\otimes\young(11,22) \; ,
\]
which agrees with the results in~\cite{LOS:2011a}.
\end{example}

\begin{example}
Consider $B^{2,3}\otimes B^{4,3}\simeq B^{4,3}\otimes B^{2,3}$ of  type $D_6^{(1)}$ and weight 
$\varpi_3+2\varpi_2+3\varpi_1$. Take the rigged configuration
\begin{center}
\unitlength 10pt
\begin{picture}(30,4.5)
\put(0,3.1){$\emptyset$}
\put(2.2,3.1){0}
\put(2.2,2.1){0}
\put(3,2){$\Yboxdim10pt\yng(2,1)$}
\put(4.2,2.1){0}
\put(5.2,3.1){0}
\put(7.2,1.1){0}
\put(7.2,2.1){0}
\put(7.2,3.1){1}
\put(8,1){$\Yboxdim10pt\yng(3,1,1)$}
\put(9.2,1.1){0}
\put(9.2,2.1){0}
\put(11.2,3.1){0}
\put(13.2,0.1){0}
\put(13.2,1.1){0}
\put(13.2,2.1){0}
\put(13.2,3.1){0}
\put(14,0){$\Yboxdim10pt\yng(3,3,1,1)$}
\put(15.2,0.1){0}
\put(15.2,1.1){0}
\put(17.2,2.1){0}
\put(17.2,3.1){0}
\put(19.2,2.1){0}
\put(19.2,3.1){0}
\put(20,2){$\Yboxdim10pt\yng(3,1)$}
\put(21.2,2.1){0}
\put(23.2,3.1){0}
\put(25.2,2.1){0}
\put(25.2,3.1){0}
\put(26,2){$\Yboxdim10pt\yng(3,1)$}
\put(27.2,2.1){0}
\put(29.2,3.1){0}
\end{picture}
\end{center}
By the combinatorial bijection $\Phi$, we obtain the two tensor products:
$$\Yvcentermath1\young(111,23\mtwo)\otimes\young(111,222,\mfour33,\mthree4\mthree)\,,\qquad
\young(111,243,4\mfour4,\mfour\mtwo\mfour)\otimes\young(111,222)\,.$$
The highest weight element corresponding to
\[
	\Yvcentermath1 \young(111,243,4\mfour4,\mfour\mtwo\mfour) \quad \text{is} \quad
	\Yvcentermath1 \young(111,222,3\mfour3,4\mthree4)= \fillmap\left( \young(111,222,3,4) \right) \; .
\]
Hence the tensor products in the KN tableaux representations are
$$\Yvcentermath1\young(111,23\mtwo)\otimes\young(111,222)\simeq \young(111,23\mtwo,4,\mfour)\otimes\young(111,222)\,,$$
which again agrees with~\cite{LOS:2011a}.
\end{example}

\begin{conjecture} \label{conjecture.energy charge}
The bijection $\Phi:(\nu,J)\longmapsto b$ preserves the statistics
\[
	\cc(\nu,J) = \Db(b).
\]
\end{conjecture}

For the highest weight element of $B^{r,s}$ this conjecture is proved in
Theorem~\ref{theorem:cc=D_for_single_rectangle}.
 

\end{document}